\documentclass[10pt]{amsart}

\usepackage{amsmath}
\usepackage{amsthm}
\usepackage{algorithm}
\usepackage[noend]{algpseudocode}
\usepackage{amsmath,bm}
\usepackage{amssymb}
\usepackage{graphicx}
\usepackage{caption}
\usepackage{color}
\usepackage{colortbl}
\usepackage{xcolor}
\usepackage{multirow}
\usepackage{array}
\usepackage{verbatimbox}
\usepackage[inner=2cm,outer=1.5cm]{geometry}
\usepackage[colorlinks=true,linkcolor=blue]{hyperref}

\newtheorem{theorem}{Theorem}[section]
\newtheorem{lemma}[theorem]{Lemma}

\theoremstyle{definition}

\theoremstyle{remark}
\newtheorem{remark}[theorem]{Remark}

\numberwithin{equation}{section}

\def\eps{\bm \epsilon}
\def\sig{\bm \sigma}
\def\bU{\boldsymbol U}
\def\bV{\boldsymbol V}
\def\bP{\boldsymbol P}
\def\bf{\boldsymbol f}
\def\bu{\boldsymbol u}
\def\bv{\boldsymbol v}
\def\bpsi{\boldsymbol \psi}
\def\bp{\boldsymbol p}
\def\bw{\boldsymbol w}
\def\bz{\boldsymbol z}
\def\bx{\boldsymbol x}

\def\bq{\boldsymbol q}

\def\bg{\boldsymbol g}
\def\be{\boldsymbol e}
\def\eu{\boldsymbol e_{\bu}}
\def\ev{\boldsymbol e_{\bv}}
\def\ep{\boldsymbol e_{\bp}}
\def\divv{\text{div}}
\def\Divv{\text{Div}}
\def\divn{\text{div }}
\def\Divu{\text{\underline{Div}}}

\def\A{\Lambda}

\def\M{M}

\def\R{A}
\def\S{S}

\definecolor{viol}{rgb}{0.75,0.15,0.95}
\definecolor{orange}{HTML}{9400D3}
\definecolor{vv}{rgb}{0.6,0.1,0.8}
\definecolor{pinkred}{rgb}{0.0,0.0,0.0}
\definecolor{pinkredd}{rgb}{0.0,0.0,0.0}
\definecolor{olive}{HTML}{6B8E23}
\definecolor{c}{HTML}{2E8B57}

\definecolor{spink}{rgb}{0.98,0.855,0.867}

\begin{document}

\title[Parameter-robust Uzawa-type iterative methods for double saddle point problems]
{Parameter-robust Uzawa-type iterative methods for double saddle point problems arising 
in Biot's consolidation and multiple-network poroelasticity models}


\author{Q. Hong}
\address{}
\curraddr{}
\thanks{}

\author{J. Kraus}
\address{}
\curraddr{}
\thanks{}

\author{M. Lymbery}
\address{}
\curraddr{}
\thanks{}

\author{F. Philo}
\address{}
\curraddr{}
\thanks{}


\date{}

\dedicatory{}

\begin{abstract}

This work is concerned with the iterative solution of systems of quasi-static multiple-network
poroelasticity (MPET) equations describing flow in elastic porous media that is permeated by
single or multiple fluid networks. Here, the focus is on a three-field formulation of the problem
in which the displacement field of the elastic matrix and, additionally, one velocity field and
one pressure field for each of the $n \ge 1$ fluid networks are the unknown physical quantities.
Generalizing Biot's model of consolidation, which is obtained for $n=1$, the MPET equations
for $n\ge1$ exhibit a double saddle point structure. 

The proposed approach is based on a framework of augmenting and splitting this three-by-three
block system in such a way that the resulting block Gauss-Seidel preconditioner defines a fully
decoupled iterative scheme for the flux-, pressure-, and displacement fields.
In this manner, one obtains an augmented Lagrangian Uzawa-type method, the analysis of which
is the main contribution of this work. The parameter-robust uniform linear convergence of this
fixed-point iteration is proved by showing that its rate of contraction is strictly less than one
independent of all physical and discretization parameters.

The theoretical results are confirmed by a series of numerical tests that compare the new fully
decoupled scheme to the very popular partially decoupled fixed-stress split iterative method, 
which decouples
only flow--the flux and pressure fields remain coupled in this case--from the mechanics problem.
We further test the performance of the block triangular preconditioner defining the new scheme
when used to accelerate the GMRES algorithm.

\end{abstract}

\maketitle

\section{Introduction}

In this paper we propose and analyze stationary iterative methods for solving the equations of multiple network
poroelastic theory (MPET) which describe flow in deformable porous media. The latter is modeled as an elastic
solid matrix comprising $n \ge 1$ superimposed fluid networks with possibly vastly varying characteristic length
scales and hydraulic conductivities, see e.g.,~\cite{Vardakis2019fluid} and the references therein.

Dual-porosity/dual-permeability models have been proposed and studied in a geomechanical context, see,
e.g.~\cite{Barenblatt1960basic,Bai_etal1993multi}, providing a generalization of Biot's consolidation model
which is obtained for $n=1$, see~\cite{Biot1941general,Biot1955theory}.
Over the last decade, the MPET equations have gradually gained attention as a tool for modeling flow across
scales and networks in soft tissue.~Biological multicompartmental poroelasticity models can be used to embed
more specific medical models, e.g., to describe water transport in the cerebral environment and explore the
pathogenesis of acute and chronic hydrocephalus~\cite{TullyVentikos2011cerebral}, or to study effects of obstructing
cerebrospinal fluid (CSF) transport and to demonstrate the impact of aqueductal stenosis and fourth ventricle
outlet obstruction (FVOO)~\cite{Vardakis2013exploring,Vardakis2016investigating}, or to find medical indications
of oedema formation~\cite{Chou2016afully}.

Recently, the MPET model has also been used in order to gain a better understanding of the
processes involved with the mechanisms behind Alzheimer's disease (AD), the most common form
of dementia~\cite{Guo_etal2018subject-specific}.
Most prominently, the so-called amyloid hypothesis states that the accumulation of neurotoxic amyloid-$\beta$
(A$\beta$) into parenchymal senile plaques or within the walls of arteries is a basic cause of this disease.
In ~\cite{Guo2019on}, a partial validation of a four-network poroelastic model for metabolic waste clearance
is presented in a qualitative way, i.e., by showing a qualitative agreement of the cerebral blood flow (CBF) data
obtained from arterial spin labeling (ASL) images and the corresponding model output for different regions of
the brain.
Although the authors of these papers conclude that there is a need for more experimental and clinical data to optimize the
boundary conditions and parameters used in numerical modeling, they also stress the potential of MPET
modeling as a testing bed for hypotheses and new theories in neuroscience research. 

Regarding the numerical solution of the MPET equations mainly two different approaches have been investigated
in the last couple of years. The first one has been proposed in~\cite{lee2018mixed} and uses a mixed finite element
formulation based on introducing an additional total pressure variable. Energy estimates for the continuous solutions
and a priori error estimates for a family of compatible semidiscretizations demonstrate that this formulation is robust
for nearly incompressible materials, small storage coefficients, and small or vanishing transfer between networks.

The second approach is based on a generalization of the classical three-field formulation of Biot's model and explicitly
accommodates Darcy's law for each fluid network.
This formulation enforces the exact conservation of mass at the price of including additionally $n$ vector fields
for the Darcy velocities (fluxes).~A parameter-robust stability analysis of this flux-based MPET model has been
presented in~\cite{Hong2018conservativeMPET} along with fully parameter-robust norm-equivalent preconditioners.
Following~\cite{HongKraus2017parameter,KanschatRiviere2017finite},
the authors propose in~\cite{Hong2018conservativeMPET} a family of strongly conservative locking-free
discretizations for the MPET model and establish the related optimal error estimates for the stationary
problems arising from implicit time discretization by the backward Euler method. These results also cover
the case of vanishing storage coefficients.

Various works can be found on discretizations and efficient iterative solvers and preconditioning techniques
for the quasi-static Biot model addressing two-field,
see, e.g.~\cite{boffi2016nonconforming, adler2017robust}, three-field,
see, e.g.,~\cite{oyarzua2016locking, Hu2017nonconforming, Lee2016parameter, HongKraus2017parameter},
and four-field formulations, see, e.g.,~\cite{Lee2016robust,Baerland2017weakly}.

Two of the most popular and likely most efficient iterative schemes for solving the equaions of
poroelasticity are the so-called undrained split and fixed-stress split iterative methods, which, contrary
to the drained split and the fixed-strain split methods, are unconditionally stable, see~\cite{Kim2011stability}.
The first convergence analysis of the former methods has been presented in~\cite{Mikelic2013convergence}
for the quasi-static Biot system.
Subsequent refined results focus mostly on variants of the fixed-stress method addressing multirate
fixed-stress split iterative schemes~\cite{Almani2016convergence}, fully discrete iterative coupling of
flow and geomechanics~\cite{almani2017convergence}, heterogenous media and linearized Biot's
equations~\cite{Both2017robust}, two-grid fixed-stress schemes for heterogeneous
media~\cite{dana2018convergence2}, or space-time finite element approximations of the quasi-static
Biot system~\cite{Bause2017space}. A strategy for optimizing the stabilization parameter in the
fixed-stress split iterative method for the Biot problem in two-field formulation has been presented
in~\cite{Storvik2018on}.

The fixed-stress method has also been recently successfully used in combination with Anderson acceleration
for the solution of non-linear poromechanics problems~\cite{Both2018Anderson}. Moreover, monolithic and splitting
based solution schemes have been considered and analyzed for solving quasi-static thermo-poroelasticity problems
with nonlinear convective transport~\cite{brun2019monolithic}. The latter work focuses on the analysis of fully and
partially decoupled schemes for heat, mechanics and flow applied to the linearized problem obtained via the so-called
$L$-scheme.~All previously mentioned works, in presence of flux and pressure unknowns, solve the flow equations
implicitely, i.e., as a coupled subsystem, a strategy which we will not pursue in this paper.

A desirable property of preconditioners, in addition to their uniformity with respect to discretization parameters, is
their robustness regarding potentially large variations of the physical parameters. This task can be studied
in the framework of operator preconditioning on the level of the continuous model,
cf.~\cite{Mardal2011preconditioning}. Targeting Biot's consolidation model the parameter-robustness of norm-equivalent preconditioners has
been established in~\cite{Lee2016parameter} for the total-pressure based formulation
and in~\cite{HongKraus2017parameter} for the classical three-field formulation based
on displacement, Darcy velocity and fluid pressure fields. Both approaches have been generalized to the
MPET model, see~\cite{lee2018mixed,Hong2018conservativeMPET}  

One potential advantage of the approach presented in~\cite{Hong2018conservativeMPET} is exact mass
conservation. A disadvantage, however, is that the presence of $n$ fluxes and $n$ associated pressures makes the system in
general more difficult and also more time-consuming to solve. 
The fixed-stress split iterative method has recently been generalized to be applicable not only to the Biot ($n=1$)
but also to the more general MPET ($n \ge 1$) systems in~\cite{Hong2018parameter-robust} which presents
a fully parameter-robust convergence analysis and determines a close to optimal acceleration parameter.

However, in the conservative approach obtained from generalizing the classical three-field formulation of Biot's model, 
the block of $n$ unknown fluxes (with $d$ components each) couples
to a block of $n$ pressure unknowns creating a subsystem with $n (d+1)$ scalar quantities of interest as compared
to the $(n (d+1) + d)$ unknown scalar functions in the whole system. Hence, considering the above-mentioned
four-network model ($n=4$) in three space dimensions ($d=3$), for example, this results in a flux-pressure
subsystem with approximately $16/19$ of the size of the whole system.~This explains why a further decoupling
of the flux from the pressure block of unknowns in an iterative method is of particular interest in this approach.

The goal of the present paper is to propose and analyze a class of fully decoupled iterative schemes, which
contrary to the fixed-stress split iterative method also decouple the {\it flux-pressure} subsystem. In this respect,
it can be seen as a continuation of the analysis presented in~\cite{Hong2018parameter-robust}.

As already mentioned, the target problem is a three-by-three block system with a double saddle point.
The abstract canonical form of the operator (matrix) of the related operator equation can be represented in the form
\begin{equation}\label{saddle-point_operator_canonical}
\begin{bmatrix}
A_1 & 0     &B_1^T\\
0     & A_2 &B_2^T\\
B_1 & B_2 & -C
\end{bmatrix}
\end{equation}
with $A_1$ and $A_2$ being symmetric positive definite (SPD) operators and $C$ a symmetric positive
semidefinite (SPSD) operator. The operator~\eqref{saddle-point_operator_canonical} defines a double saddle point
problem and can be rearranged in such a way that it has the form
\begin{equation}\label{saddle-point_operator_tridiag}
\begin{bmatrix}
A_1 & B_1^T & 0 \\
B_1 & -C &B_2\\
0 & B_2^T & A_2
\end{bmatrix}
\end{equation}
and thus fits the definition of a multiple saddle point operator as given in~\cite{sogn2019schur} where
block-diagonal Schur complement preconditioners for multiple saddle point problems of block tridiagonal
form are analyzed.
We will use a combined augmentation and splitting technique to construct in a block Gauss-Seidel framework
fully decoupled augmented Lagrangian Uzawa-type methods for linear systems with an operator (matrix) of the
canonical form~\eqref{saddle-point_operator_canonical}.
Although our methodical approach to construct preconditioners is similar to the one taken in the recent
works~\cite{Benzi2018iterative,Benzi2019UzawatypeAA}, see also~\cite{white2016block}, there are also
major differences. Firstly, the double saddle point problems considered
in~\cite{Benzi2018iterative,Benzi2019UzawatypeAA} are generated by operators of the canonical form
\begin{equation}\label{saddle-point_operator_Benzi}
\begin{bmatrix}
A_1 & B_1^T &B_2^T\\
B_1 & 0 & 0\\
B_2 & 0 & -C
\end{bmatrix}
\end{equation}
with $A_1$ being SPD and $C$ being SPSD. It can easily be seen that the
operators~\eqref{saddle-point_operator_canonical} and \eqref{saddle-point_operator_Benzi}
are of a different type in the sense that they can not be transferred one into the other by
permutations of rows and columns. The second main difference is that the analysis
in~\cite{Benzi2018iterative,Benzi2019UzawatypeAA} uses arguments from classical linear
algebra whereas our convergence proofs use techniques from functional analysis aiming at
quantitative bounds that might be useful when applying the proposed iterative methods at
the level of finite element approximations of the continuous problems.

The remainder of the paper is organized as follow: In Section~\ref{sec:Section_2}, we first
formulate the MPET problem, introduce the notation and transform the problem into a coupled
system with a double saddle point operator of the form~\eqref{saddle-point_operator_tridiag}.
Based on this notation we then recall the fixed-stress split iterative method in a block Gauss-Seidel
framework. It follows the construction of a new class of fully decoupled iterative Uzawa-type methods,
which requires an additional augmentation step. This section ends with summarizing some preliminary
and auxiliary results that are used in the convergence analysis of the new class of methods presented 
in Section~\ref{sec:convergence}. The numerical tests in Section~\ref{sec:Section_4} serve
the assessment of the performance of the iterative methods and preconditioners developed in this
paper comparing them also with the fixed-stress split iterative method analyzed
in~\cite{Hong2018parameter-robust}.

\section{Iterative coupling methods for the MPET problem}\label{sec:Section_2}

\subsection{The MPET system -  formulation and notation}
%
Consider the quasi-static MPET equations in a bounded Lipschitz domain $\Omega \subset \mathbb R^d$, $d=2,3$:
\begin{subequations}\label{eq:MPET}
\begin{align}
 \bv_i + K_i \nabla p_i &= \bm 0\;\; \text{in}~~ \Omega\times (0,T),~~ i=1,\ldots,n,
\label{MPET2} \\
-\alpha_i \divv \dot{\bu} - \divv \bv_i - c_{p_{i}} \dot{p}_{i} - \sum_{\substack{j=1\\j\neq i}}^{n}\beta_{ij} (p_i-p_j) &=g_i\;\;\text{in}~~ \Omega\times (0,T),~~  i=1,\ldots,n, \label{MPET3}\\
-\divv \sig + \sum_{i=1}^{n}\alpha_i \nabla p_i  &= \bf~~ \text{in}~~ \Omega\times (0,T).\label{MPET1}
\end{align}
\end{subequations}
The unknown physical quantities in this system are the displacement field $\bu$, the seepage velocities, or fluxes, $\bv_i$,
and the scalar pressure fileds $p_i$. 
The effective stress and strain tensors are given by
\begin{equation}
	\sig = 2\mu \eps(\bu) + \lambda \text{div}(\bu)\bm I \quad \text{and}\quad 
	\eps(\bu)  = \frac{1}{2}(\nabla \bu + (\nabla \bu)^T), \label{constitutive_compatibility}
\end{equation}
respectively with the Lam\'e parameters $\lambda$ and $\mu$ defined via the modulus of elasticity $E$ and the Poisson ratio $\nu\in[0,1/2)$ 
as follows: $$\lambda:=\frac{\nu E}{(1+\nu)(1-2\nu)}, \qquad\mu:=\frac{E}{2(1+\nu)}.$$

In~\eqref{eq:MPET}, $\alpha_i$ denote the Biot-Willis coefficients, 
$K_i$ the hydraulic conductivities, which in this paper for convenience only, are scalars defining the tensor coefficients $\bm K_i = K_i \bm I$,
$c_{p_{i}}$ the constrained specific storage coefficients.  
Considering the right-hand sides in~\eqref{MPET1} and~\eqref{MPET3},
$\bf$ denotes the body force density
whereas 
$g_i$ represent the fluid extractions or injections, see e.g.~\cite{Showalter2010poroelastic} and the references therein. 
The parameters $\beta_{ij}=\beta_{ji}$, $i\neq j$ couple the network pressures and are called network transfer coefficients. 

By substituting the expression for the stress tensor from~\eqref{constitutive_compatibility} in \eqref{MPET1} the MPET system takes the form: 
\begin{subequations}\label{eq:MPET:vpu}
\begin{align}
\bv_i +K_i \nabla p_i={\mathbf 0},\;\; ~~ i=1,\ldots,n,  
\label{MPET2vpu} \\
- \divv \bv_i - c_{p_{i}} \dot{p}_{i} - \sum_{\substack{j=1\\j\neq i}}^{n}\beta_{ij} (p_i-p_j) -\alpha_i \divv \dot{\bu} &=g_i,\;\;~~  i=1,\ldots,n,
\label{MPET3vpu}\\
 \sum_{i=1}^{n}\alpha_i \nabla p_i  -2\mu\divv \eps(\bu)-\lambda\nabla \divv \bu &= \bf.\label{MPET1vpu}
\end{align}
\end{subequations} 

After imposing proper boundary and initial conditions, see~\cite{Hong2018conservativeMPET}, 
and using the backward Euler method for 
time discretization,
one has to solve a static problem of the form 
\begin{subequations}\label{eq:1}
	\begin{align}
	 K^{-1}_i\bv_i ^k+  \nabla p_i^k &= {\mathbf 0}, \qquad i=1,\dots,n,  \label{eq:1b}\\
	-\alpha_i \divv\bu^k- \tau \divv\bv_i^k -c_{p_i}p_i^k-\tau \sum_{\substack{{j=1}\\j\neq i}}^{n}\beta_{ij}(p_i^k -p_j^k) &= g_i^k, \qquad i=1,\dots,n, \label{eq:1c}\\
	-2\mu\divv \eps(\bu^k)  - \lambda \nabla \divv \bu^k +\sum_{i=1}^{n}\alpha_i \nabla p_i^k &= \bf^k, \label{eq:1a}
	\end{align}
\end{subequations}
in each time step, i.e., at every time moment $t_k = t_{k-1}+\tau$, $k=1,2,\ldots$. Here,  
$\bu^k$, $\bv_i^k$, $p_i^k$ are approximations
of $\bu$,  $\bv_i$, $p_i$ at $t = t_k$ and
$\bf^k = \bf(x, t_k)$, ${g}^k_i = -\tau g_i(x, t_k) - \alpha_i \divv(\bu^{k-1}) - c_{p_i}p_i^{k-1}$ for $i=1,\ldots,n$.  
After  
dividing~\eqref{eq:1} by $2\mu$, denoting 
$$
\frac{\lambda}{2\mu}\rightarrow \lambda, \ \frac{\alpha_i}{2\mu}\rightarrow \alpha_{i},\ 
\frac{\bm f^k}{2\mu}\rightarrow \bm f^k, \ \frac{\tau}{2\mu}\rightarrow \tau, \ 
\frac{c_{p_i}}{2\mu}\rightarrow c_{p_i}, \ \frac{g_i^k}{2\mu} \rightarrow g_i^k,\quad i=1,\ldots,n,
$$
and further introducing the new variables  
$$
\bv_i:=\frac{\tau}{\alpha_i} \bv_i^k ,\quad p_i:=\alpha_i p_i^k  , \quad \bu:=\bu^k,\quad 
\bf:=\bf^k, \quad g_i:=\frac{g_i^k}{\alpha_i}  ,\quad i=1,\ldots,n,
$$
system~\eqref{eq:1} can be 
presented in the form
\begin{subequations}\label{eq:4}
	\begin{align}
	\tau^{-1} K^{-1}_i\alpha_i^2{\bv}_i + \nabla {p}_i &={\mathbf 0}, \qquad i=1,\dots,n,\\
	-\divv \bu - \divv{\bv}_i  -\frac{c_{p_i}}{\alpha_i^2} {p}_i + \sum_{\substack{{j=1}\\j\neq i}}^{n} \left(-\frac{\tau\beta_{ij}}{\alpha_i^2} {p}_i+\frac{\tau\beta_{ij}}{\alpha_i\alpha_j} {p}_j\right) 
	&= {g}_i,\qquad i=1,\dots,n, \\
	-\divv \eps(\bu) - \lambda \nabla \divv \bu +\sum_{i=1}^{n} \nabla {p}_i &= \bf,
	\end{align}
\end{subequations}
where we have also multiplied~\eqref{eq:1b} by $\alpha_i$ and \eqref{eq:1c} by $\alpha_i^{-1}$. 
In what follows we will also make use of the notation 
${\boldsymbol v}^T:=({\boldsymbol v}_1^T,\dots,{\boldsymbol v}_n^T)$, ${\boldsymbol z}^T:
=({\boldsymbol z}_1^T,\dots,{\boldsymbol z}_n^T)$, 
${\boldsymbol p}^T:=(p_1,\dots,p_n)$, ${\boldsymbol q}^T:=(q_1,\dots,q_n)$ where 
${\boldsymbol v}, {\boldsymbol z} \in {\boldsymbol V}={\boldsymbol V}_1\times \dots\times {\boldsymbol V}_n$,
${\boldsymbol p}, {\boldsymbol q} \in {\boldsymbol P}=P_1\times \dots\times P_n$ 
and ${\boldsymbol U} \hspace{-0.4ex}=\hspace{-0.4ex} \{ {\boldsymbol u} \in H^1(\Omega)^d:\hspace{-0.4ex} {\boldsymbol u} \hspace{-0.4ex}= {\boldsymbol 0}
\text{ on } \Gamma_{{\boldsymbol u},D} \},$
${\boldsymbol V}_i \hspace{-0.4ex}=\hspace{-0.4ex} \{ {\boldsymbol v}_i\in H({\text{div}},\Omega):
{\boldsymbol v}_i \cdot {\boldsymbol n} =0  \text{ on } \Gamma_{p_i,N} \}$, $P_i = L^2(\Omega)$, and $P_i = L^2_0(\Omega)$
if $\Gamma_{{\boldsymbol u},D} = \Gamma=\partial \Omega$.
Using the parameter substitutions 
$$R^{-1}_i := \tau^{-1}K_i^{-1}\alpha_i^2,\quad \alpha_{p_i} := 
\frac{c_{p_i}}{\alpha_i^2},\quad \beta_{ii} := \sum_{\substack{j=1\\j\neq i}}^{n}\beta_{ij},\quad
{\alpha}_{ij} := \frac{\tau \beta_{ij}}{\alpha_i \alpha_j},\quad \tilde\alpha_{ii}:=- \alpha_{p_i} - \alpha_{ii}$$ for $i,j=1,\dots,n,$
we further rewrite system~\eqref{eq:4} as

\begin{align}\label{MPET_strong}
{\mathcal A}\left(\begin{array}{c} \bv \\ \bp \\ \bu \end{array} \right) =
&\begin{bmatrix}
A_v&B^T_v&0\\
B_v&-C&B_u\\
0&B_u^T&A_u
\end{bmatrix}
\left(\begin{array}{c} \bv \\ \bp \\ \bu \end{array} \right)=\left(\begin{array}{c}\boldsymbol 0\\ \bg\\\bf\end{array} \right) 
\end{align}
where
\begin{align*}
&A_v:=\hspace{-1.ex}\begin{bmatrix}
R^{-1}_1I &     0      & \dots  &0         \\
  0       &     \ddots &     & \vdots       \\
  \vdots  &            &\ddots  & 0\\
   0      &  \dots     & 0      & R^{-1}_nI
\end{bmatrix}\hspace{-1.ex},\hspace{1.ex}
B_v:=\hspace{-1.ex}\begin{bmatrix}
-  \divv&0        &\dots   &  0 \\
0       &  \ddots &        & \vdots  \\
 \vdots &         & \ddots & 0 \\
   0    &     \dots &   0    &-  \divv 
\end{bmatrix}\hspace{-1.ex},\hspace{1.ex}
B_u:=\hspace{-1.ex}\begin{bmatrix}
- \divv  \\
\vdots \\
\vdots\\
- \divv
\end{bmatrix}\hspace{-1.ex},\hspace{1.ex}
-C:=\hspace{-1.ex}\begin{bmatrix}
\tilde\alpha_{11}I& \alpha_{12} I & \dots& \alpha_{1n}I\\
 \alpha_{21}I  & \ddots & & \alpha_{2n}I\\
 \vdots &  & \ddots & \vdots  \\
  \alpha_{n1}I  &  \alpha_{n2}I &   \dots   & \tilde\alpha_{nn} I \\ 
\end{bmatrix}\\
&\hbox{and} ~A_u:=-\divv \eps  - \lambda \nabla \divv.  
\end{align*}
%
For the scaled parameters, 
we make the rather non-restrictive assumptions
\begin{align}\label{parameter:range}
\lambda \ge 0,\quad  R^{-1}_1 ,\dots, R^{-1}_n > 0 ,\quad \alpha_{p_1},\dots,\alpha_{p_n} \ge 0,
\quad \alpha_{ij}\ge 0, ~~~i, j=1,\dots, n.
\end{align}

From now on, we will use the same symbols for denoting operators and their corresponding coefficient matrices.
%
Additionally, let us introduce
$$
\Lambda _{1}:=  
\begin{bmatrix}
\alpha_{11} & -\alpha_{12} & \dots &-\alpha_{1n}  \\
-\alpha_{21} & \alpha_{22} & \dots &-\alpha_{2n}  \\
\vdots & \vdots & \ddots & \vdots  \\
-\alpha_{n1} & -\alpha_{n2} & \dots &\alpha_{nn}  
\end{bmatrix},\quad
\Lambda _2:=
\begin{bmatrix}
{\alpha_{p_1}} &0&\dots &0\\
0&{\alpha_{p_2}}&\dots &0\\
\vdots&\vdots &\ddots&\vdots\\
0&0&\dots&  {\alpha_{p_n}}
\end{bmatrix},
$$
i.e. $C=\Lambda_1+\Lambda_2$. Further, denote
$R^{-1}:=\max\{R^{-1}_i:i=1,2,\cdots,n\}, \lambda_0:=\max\{1,\lambda\}$,  
\begin{align*}
\text{$\A _3$}&:=
\begin{bmatrix}
R &0&\dots &0\\
0&\ddots&{\ddots}&\vdots\\
\vdots&{\ddots}&\ddots&\vdots\\
0&\dots&0& R
\end{bmatrix}, \quad
\qquad \A _{{4}}:=
\begin{bmatrix}
\frac{1}{\lambda_0} &\dots& {\dots} & \frac{1}{\lambda_0}\\
\vdots& & &\vdots\\
\vdots& & &\vdots\\
\frac{1}{\lambda_0}&\dots & {\dots}&\frac{1}{\lambda_0}
\end{bmatrix},
\end{align*}
$$\Lambda:=\Lambda_1+\Lambda_2+\Lambda_3+\Lambda_4,$$
and also, for any block vector $\bz$ and vector $\bu$ 
$$
\Divv\bz:= \left(\begin{array}{c}\divv\bz_1 \\ \ \vdots \\ \divv\bz_n \end{array} \right),\quad
\Divu\bu:= \left(\begin{array}{c}\divv\bu \\ \ \vdots \\ \divv\bu \end{array} \right).
$$

\subsection{The fixed-stress split iterative method revisited}
For any operator $\Lambda_L:  {\boldsymbol P}\rightarrow  {\boldsymbol P}^*$, $\mathcal A$ can be decomposed as follows: 
\begin{align}
{\mathcal A}=
\begin{bmatrix}
A_v&B^T_v&0\\
B_v&-C-\Lambda_L&0\\
0&B_u^T&A_u
\end{bmatrix}
+
\begin{bmatrix}
0&0&0\\
0&\Lambda_L&B_u\\
0&0&0
\end{bmatrix}
\end{align}
Applying the block Gauss-Seidel method to the above system, we obtain 
\begin{align}
\begin{bmatrix}
A_v&B^T_v&0\\
B_v&-C-\Lambda_L&0\\
0&B_u^T&A_u
\end{bmatrix}
\left(\begin{array}{c} \bv^{k+1} \\ \bp^{k+1} \\ \bu^{k+1} \end{array} \right)
+
\begin{bmatrix}
0&0&0\\
0&\Lambda_L&B_u\\
0&0&0
\end{bmatrix}
\left(\begin{array}{c} \bv^{k} \\ \bp^{k} \\ \bu^{k} \end{array} \right)
=\left(\begin{array}{c}\boldsymbol 0\\ \bg\\\bf\end{array} \right)
\end{align}
or, equivalently,  
\begin{align}\label{fixed_stress_sys}
\begin{bmatrix}
A_v&B^T_v&0\\
B_v&-C-\Lambda_L&0\\
0&B_u^T&A_u
\end{bmatrix}
\left(\begin{array}{c} \bv^{k+1} \\ \bp^{k+1} \\ \bu^{k+1} \end{array} \right)
=\left(\begin{array}{c}\boldsymbol 0\\ \bg\\\bf\end{array} \right)
-
\begin{bmatrix}
0&0&0\\
0&\Lambda_L&B_u\\
0&0&0
\end{bmatrix}
\left(\begin{array}{c} \bv^{k} \\ \bp^{k} \\ \bu^{k} \end{array} \right)
\end{align}
which is (a block variant of) the fixed-stress method. 
In~\cite{Hong2018parameter-robust} a parameter-robust convergence analysis of this method has been presented for 
the choice 
\begin{align}
\Lambda_L=L 
\begin{bmatrix}
I &I & \dots& I\\
I & \ddots & &I\\
  \vdots &  & \ddots & \vdots\\
I  &  I   &   \dots   & I    
\end{bmatrix}
\quad \hbox{where}\quad~L\ge \frac{1}{\lambda+c_K^2}, 
\end{align}
and $c_K$ is the constant in the estimate
\begin{equation}\label{ck}
\|\eps(\bw)\|\ge c_{K}\|\divv \bw \| \quad \text{for all }\bw \in \bU
\end{equation}
 where $\| \cdot \|$ denotes the $L^2$ norm, on the left-hand side of~\eqref{ck} of a tensor-valued
and on the right-hand side of a scalar-valued function.
Note that~\eqref{ck} holds true for example for $c_K=1/\sqrt{d}$ where $d$ is the space dimension. 

\subsection{Uzawa-type methods in block Gauss-Seidel framework}
Now for any positive definite operator $M: {\boldsymbol P}^*\rightarrow  {\boldsymbol P}$, we consider 
the equivalent augmented MPET system 
\begin{align}\label{Augweak}
\hat {\mathcal A}\left(\begin{array}{c} \bv \\ \bp \\ \bu \end{array} \right) =
&\begin{bmatrix}
A_v + B^T_v \M B_v&B^T_v-B^T_v \M C&B^T_v \M B_u\\
-B_v&C&-B_u\\
0&B_u^T&A_u
\end{bmatrix}
\left(\begin{array}{c} \bv \\ \bp \\ \bu \end{array} \right)=\left(\begin{array}{c} B^T_v \M\bg\\ -\bg\\\bf\end{array} \right).
\end{align}
Further, for any positive definite operator $S: {\boldsymbol P}\rightarrow  {\boldsymbol P}^*$, 
we decompose $\hat {\mathcal A}$ in the form
\begin{align}
\hat {\mathcal A}
=
&\begin{bmatrix}
A_v + B^T_v \M B_v&0&0\\
-B_v&S&0\\
0&B_u^T&A_u
\end{bmatrix}
+
\begin{bmatrix}
0&B^T_v-B^T_v \M C&B^T_v \M B_u\\
0&-S+C&-B_u\\
0&0&0
\end{bmatrix}.
\end{align}
Next, applying the block Gauss-Seidel method to the above system yields
\begin{align}
&\begin{bmatrix}
A_v + B^T_v \M B_v&0&0\\
-B_v&S&0\\
0&B_u^T&A_u
\end{bmatrix}
\left(\begin{array}{c} \bv^{k+1} \\ \bp^{k+1} \\ \bu^{k+1} \end{array} \right)
+
\begin{bmatrix}
0&B^T_v-B^T_v \M C&B^T_v \M B_u\\
0&-S+C&-B_u\\
0&0&0
\end{bmatrix}
\left(\begin{array}{c} \bv^{k} \\ \bp^{k} \\ \bu^{k} \end{array} \right)
=\left(\begin{array}{c} B^T_v \M\bg\\ -\bg\\\bf\end{array} \right),
\end{align}
namely 
\begin{align}\label{GSsystem}
&\begin{bmatrix}
A_v + B^T_v \M B_v&0&0\\
-B_v&S&0\\
0&B_u^T&A_u
\end{bmatrix}
\left(\begin{array}{c} \bv^{k+1} \\ \bp^{k+1} \\ \bu^{k+1} \end{array} \right)
=\left(\begin{array}{c} B^T_v \M\bg\\ -\bg\\\bf\end{array} \right)
-
\begin{bmatrix}
0&B^T_v-B^T_v \M C&B^T_v \M B_u\\
0&-S+C&-B_u\\
0&0&0
\end{bmatrix}
\left(\begin{array}{c} \bv^{k} \\ \bp^{k} \\ \bu^{k} \end{array} \right).
\end{align}
System~\eqref{GSsystem} can be expressed in terms of bilinear forms as follows: 
\begin{algorithm}[H]
	\caption{
	Fully decoupled iterative scheme for weak flux-pressure-displacement formulation of MPET problem}
	\begin{algorithmic}[1]
 \Statex{\textbf{Step a:} Given $\bp^k$ and $\bu^k$, we first solve for $\bv^{k+1}$, such that for all ${\boldsymbol z} \in {\boldsymbol V}$ there holds}
\begin{align}
&(A_v\bv^{k+1},\bz)  +(\M\Divv \bv^{k+1},\Divv \bz) =- (\M\bg,\Divv \bz) + (\bp^{k},\Divv \bz) - (\M(\A_1+\A_2)\bp^k,\Divv \bz)-  (\M\Divu \bu^k,\Divv \bz).
\end{align}
 \Statex{\textbf{Step b:} Given $\bu^k$ and $\bv^{k+1}$, we solve for $\bp^{k+1}$, such that for all ${\boldsymbol q} \in {\boldsymbol P}$ there holds}
 \begin{align}
(\S\bp^{k+1},\bq)=-(\bg,\bq)+(\S\bp^k,\bq)- ((\A_1 +\A_2)\bp^k,\bq))-(\Divu\bu^k,\bq) - (\Divv\bv^{k+1},\bq) .
 \end{align}
 \Statex{\textbf{Step c:} Given $\bp^{k+1}$ and $\bv^{k+1}$, we solve for $\bu^{k+1}$, such that for all ${\boldsymbol w} \in {\boldsymbol U}$ there holds}
 \begin{align}
&(\eps(\bu^{k+1}), \eps(\bw))+ \lambda  (\divv\bu^{k+1},\divv \bw) = (\bf,\bw) + (\bp^{k+1},\Divu \bw).
 \end{align}
\end{algorithmic}
\label{alg1}
\end{algorithm}

\subsection{Preliminary results}
We first present a result from linear algebra which will be useful in the proof of Theorem~\ref{theorem_errors_vu} 
in Section~\ref{sec:convergence}.

\begin{lemma}\label{etauu:7}
For any $a>0$ and $b>0$, denote $\boldsymbol e=(\underbrace{1,\dots,1}_n)^T$ and $(a I_{n\times n}+b \boldsymbol e\boldsymbol e^T)^{-1}=(b_{ij})_{n\times n}$. Then we have that
\begin{equation}
0<\sum_{\substack{i=1}}^n\sum_{\substack{j=1}}^n b_{ij}=\frac{n}{(a+nb)}.
\end{equation}
\end{lemma}
\begin{proof}
The proof is based on the Sherman-Morrison-Woodbury formula and follows the arguments of the proof
of Lemma 1 in \cite{Hong2018conservativeMPET}. 
\end{proof}

Next, let us recall some well known results, see~\cite{ Brezzi1974existence,Boffi2013mixed}.
\begin{lemma} There exists a constant $\beta_s > 0$ such that:
\begin{align}
	\inf_{(q_1,\cdots,q_n)\in P_1\times\cdots\times P_n}
\sup_{{\boldsymbol u}\in {\boldsymbol U}}
\frac{\left({\rm div} {\boldsymbol u}, \sum\limits_{i=1}^n q_i\right)}{\|{\boldsymbol u}\|_1\left\|\sum\limits_{i=1}^n q_i\right\|} \geq \beta_s  
	\end{align}
	\label{lemma1}	
\end{lemma}
\begin{lemma} There exists a constant $\beta_d > 0$ such that: 
	\begin{align}
	\inf_{q\in P_i} \sup_{{\boldsymbol v} \in {\boldsymbol V}_i} \frac{({\rm div} {\boldsymbol v}, q)}{\|{\boldsymbol v}\|_{\rm div}\|q\|}
\geq \beta_d, \quad i=1,\dots,n. 
	\end{align}
	\label{lemma2}
\end{lemma}
Here $\| \cdot \|_1$ and $\|{\boldsymbol v}\|_{\rm div}$ denote the standard $H^1$ and $H({\rm div})$ norms of vector-valued functions,
respectively, i.e., $\|{\boldsymbol u}\|_1^2 := \int_\Omega \nabla \boldsymbol u : \nabla \boldsymbol u + \boldsymbol u \cdot \boldsymbol u \, dx$
and $\|{\boldsymbol v}\|_{\rm div}^2 := \int_\Omega \divv {\boldsymbol v} \, \divv {\boldsymbol v} + \boldsymbol v \cdot \boldsymbol v \, dx$.

Our task will be to study the errors 
\begin{subequations}\label{error}
\begin{eqnarray}
\be_u^k & = & \bm u^k-\bm u\in \bU, \\
\be_{v_i}^k&=&\bm v_i^k-\bm v_i \in \bV_i, \quad i=1,\ldots,n, \\
e_{p_i}^k&=&p_i^k-p_i \in P_i, \quad i=1,\ldots,n, 
\end{eqnarray}
\end{subequations}
of the $k$-th iterates $\bm u^k$, $\bm v^k_i$, $p^k_i$, $i=1,\ldots,n$, generated by Algrorithm~\ref{alg1}. 
For that reason, 
we consider  
the following error equations
\begin{subequations}
\begin{align}
&(A_v\ev^{k+1},\bz) - (\ep^{k},\Divv \bz) + (\M\Divu \eu^{k},\Divv \bz) + (\M\Divv \ev^{k+1},\Divv \bz) + (\M(\A_1+\A_2)\ep^k,\Divv \bz)=0,\label{uzawaq1b}\\
&(\S\ep^{k+1},\bq)-(\S\ep^{k},\bq) + (\Divu\eu^{k},\bq) +(\Divv\ev^{k+1},\bq) + ((\A_1 +\A_2)\ep^{k},\bq)=0,\label{uzawaq1c}\\
&(\eps(\eu^{k+1}), \eps(\bw))+ \lambda (\divv\eu^{k+1},\divv \bw) -(\ep^{k+1},\Divu \bw) = 0\label{uzawaq1a}
\end{align}\label{uzawaq1}
\end{subequations}
where the error block-vectors $\be_v^k$ 
and $\be_p^k$ are given by 
$(\be_{\bv}^k)^T = ((\be_{v_1}^k)^T,\ldots,(\be_{v_n}^k))^T$, $(\be_{\bp}^k)^T = (e_{p_1}^k,\ldots,e_{p_n}^k)$.

To complete the design of Algorithm \ref{alg1}, we need to specify $M$ and $S$.   
By Lemma \ref{lemma2} we have that for all
$\be_{p_i}^{k+1} \in P_i$ there exists $\bpsi_i \in \bV_i$ such that 
$\divv\bpsi_i =\be_{p_i}^{k+1}$ and $\|\bpsi_i\|_{\divn} \leq { \beta_d^{-1}\|\be_{p_i}^{k+1}\|}$ for all
$i=1,\dots,n$, i.e., $\Divv\bpsi =\ep^{k+1}$ and  $\|\bpsi\|_{\divn} \leq { \beta_{v}^{-1}\|\ep^{k+1}\|}.$
Setting $\bq= \S^{-1}\ep^{k+1}$ in \eqref{uzawaq1c} and $\bz=\bpsi$ in \eqref{uzawaq1b}, from $\Divv \bpsi=\ep^{k+1}$
it follows that
\begin{subequations}
	\begin{align}
	&(A_v\ev^{k+1},\bpsi) - (\ep^{k}, \ep^{k+1}) + (\M\Divu \eu^{k+1},\ep^{k+1}) +(\M\Divv \ev^{k+1},\ep^{k+1}) + (\M(\A_1+\A_2)\ep^k, \ep^{k+1})=0,\label{uzawaq4a}\\
	&(\ep^{k+1},\ep^{k+1})-(\ep^{k},\ep^{k+1})+(\S^{-1}\Divu\eu^{k}, \ep^{k+1}) +(\S^{-1}\Divv\ev^{k+1}, \ep^{k+1}) +(\S^{-1}(\A_1 +\A_2)\ep^{k}, \ep^{k+1})=0.\label{uzawaq4b}
	\end{align}\label{uzawaq4}
\end{subequations}
Subtracting \eqref{uzawaq4a} from \eqref{uzawaq4b} yields
\begin{align*}
\|\ep^{k+1}\|^2&=(A_v\ev^{k+1},\bpsi)-((\S^{-1}-\M)(\Divu\eu^{k} +\Divv\ev^{k+1} +(\A_1 +\A_2)\ep^{k}), \ep^{k+1}),
\end{align*} 
implying 
\begin{align*}
\|\ep^{k+1}\|^2&\le\|A_v^{\frac{1}{2}}\ev^{k+1}\|\|A_v^{\frac{1}{2}}\bpsi\|+\|(\S^{-1}-\M)(\Divu\eu^{k+1} +\Divv\ev^{k+1} +(\A_1 +\A_2)\ep^{k})\| \|\ep^{k+1}\|\\
&\le\sqrt{R^{-1}}\|\R_v^{\frac{1}{2}}\ev^{k+1}\|\|\bpsi\|+\|(\S^{-1}-\M)(\Divu\eu^{k} +\Divv\ev^{k+1} +(\A_1 +\A_2)\ep^{k})\| \|\ep^{k+1}\|\\
&\le\beta_d^{-1}\sqrt{R^{-1}}\|\R_v^{\frac{1}{2}}\ev^{k+1}\|\|\ep^{k+1}\|+\|(\S^{-1}-\M)(\Divu\eu^{k} +\Divv\ev^{k+1} +(\A_1 +\A_2)\ep^{k})\| \|\ep^{k+1}\|.
\end{align*} 
We conclude that 
\begin{align}\label{SM}
\|\ep^{k+1}\|&\le\beta_d^{-1}\sqrt{R^{-1}}\|\R_v^{\frac{1}{2}}\ev^{k+1}\|+\|(\S^{-1}-\M)(\Divu\eu^{k} +\Divv\ev^{k+1} +(\A_1 +\A_2)\ep^{k})\|.
\end{align} 
Estimate \eqref{SM} suggests choosing $\S=\M^{-1}$ in order to minimize the upper bound for $\|\ep^{k+1}\|$. This results in the 
following statement. 
\begin{lemma}
Consider Algorithm \ref{alg1} and let $\S=\M^{-1}$, then we have
\begin{align}
\text{$\|\R_v^\frac{1}{2}\ev^{k+1}\|^2 \geq R\beta_d^2\|\ep^{k+1}\|^2 =\beta_d^2\|\A_3^{\frac{1}{2}}\ep^{k+1}\|^2.$}
\label{u3}
\end{align}

\end{lemma}
\noindent
The relationship $S=M^{-1}$ reduces our design task to the determination of either $S$ or $M$. 
%
In the remainder of this paper, we analyze and numerically test Algorithm~\ref{alg1} for the specific choice
\begin{equation}\label{defS}
S:=\A_1+\A_2+L_1\A_3+L_2\A_4,
\end{equation}
where $L_1$ and $L_2$ are scalar parameters which are later to be determined.

\section{Convergence theory of Uzawa-type algorithms for MPET}\label{sec:convergence}	

This section is devoted to the convergence analysis of Algorithm~\ref{alg1}. Our aim is to establish a uniform 
bound on the convergence rate, i.e., a bound independent 
of any model and discretization parameters. 

We start with deriving some useful 
auxiliary results presented 
in the following two lemmas. These afterwards assist us in establishing a parameter-robust upper bound on the pressure error in 
a weighted norm.

\begin{lemma}
Considering Algorithm~\ref{alg1} with $S$ as defined in~\eqref{defS},
 the errors $\be_{\bu}^k$, $\be_{\bv}^k $ and $ \be_{\bp}^k$ defined in \eqref{error} \label{alg3le1} satisfy the following estimate: 
\begin{align}
&\frac{1}{2}\|\eps(\eu^{k+1})\|^2+\frac{\lambda}{2}\|\mathrm{div}\eu^{k+1}\|^2+ \|\R_v^{\frac{1}{2}}\ev^{k+1}\|^2+\|(\A_1 +\A_2)^{\frac{1}{2}}\ep^{k+1}\|^2
+\frac{L_1}{2}\|\A_3^{\frac{1}{2}}\ep^{k+1}\|^2+ \frac{L_2}{2} \|\A_4^{\frac{1}{2}}\ep^{k+1}\|^2\nonumber\\&\leq\frac{L_1}{2}\|\A_3^{\frac{1}{2}}\ep^k\|^2
+\frac{L_2}{2}\|\A_4^{\frac{1}{2}}\ep^{k}\|^2+\left(\frac{\lambda_0}{2(c_K^2+\lambda)}-\frac{L_2}{2}-\frac{L_1R\lambda_0}{2n}\right)\| \A_4^{\frac{1}{2}}(\ep^{k+1}-\ep^k)\|^2.\label{3last}
\end{align}
\end{lemma}
\begin{proof}
	By setting $\bq= \M\Divv \ev^{k+1}$ in \eqref{uzawaq1c} and $\bz=\ev^{k+1}$ in \eqref{uzawaq1b}
	we obtain
		\begin{align*}%
		&(\R_v\ev^{k+1},\ev^{k+1}) - (\ep^{k},\Divv \ev^{k+1}) + (\M\Divu \eu^{k},\Divv \ev^{k+1}) +(\M\Divv \ev^{k+1},\Divv \ev^{k+1}) + (\M(\A_1+\A_2)\ep^k,\Divv \ev^{k+1})=0,
		\\
		&(\ep^{k+1},\Divv\ev^{k+1})=(\ep^{k},\Divv\ev^{k+1})
		- (\M\Divu\eu^{k},\Divv\ev^{k+1}) - (\M\Divv\ev^{k+1},\Divv\ev^{k+1}) 
		- (\M(\A_1 +\A_2)\ep^{k},\Divv\ev^{k+1})
		\end{align*}
	from where it immediately follows that
	\begin{align}
	(\ep^{k+1},\Divv\ev^{k+1})=(\R_v\ev^{k+1},\ev^{k+1}).
	\label{3u1}
	\end{align}  
Choosing $\bq= \ep^{k+1}$ in \eqref{uzawaq1c} and
	$\bw=\eu^{k+1}$ in \eqref{uzawaq1a} yields
	\begin{subequations}
		\begin{align}
		&(\eps(\eu^{k+1}), \eps(\eu^{k+1}))+ \lambda  (\divv\eu^{k+1},\divv \eu^{k+1}) -(\ep^{k+1},\Divu \eu^{k+1})= 0,\label{3uzawaq3a}\\
		&(S\ep^{k+1},\ep^{k+1})=(S\ep^{k},\ep^{k+1})- (\Divu\eu^{k},\ep^{k+1}) 
		- (\Divv\ev^{k+1},\ep^{k+1}) - ((\A_1 +\A_2)\ep^{k},\ep^{k+1}).\label{3uzawaq3b}
		\end{align}\label{3uzawaq3}
	\end{subequations}
	
	\noindent	
Next, summing \eqref{3uzawaq3a} and \eqref{3uzawaq3b} and applying \eqref{3u1} it follows that
	\begin{align}
	\|\eps(\eu^{k+1})\|^2+\lambda\|\divv\eu^{k+1}\|^2+\|S^{\frac{1}{2}}\ep^{k+1}\|^2-((L_1\A_3 +L_2\A_4)\ep^{k},\ep^{k+1}) 
	= &(\Divu\eu^{k+1}-\Divu\eu^{k},\ep^{k+1}) -\|\R_v^{\frac{1}{2}}\ev^{k+1}\|^2.\label{3ee}
	\end{align}
     In order to simplify \eqref{3ee} we first rewrite $\| S^{\frac{1}{2}}\ep^{k+1}\|^2-((L_1\A_3 +L_2\A_4)\ep^{k},\ep^{k+1})$, 
     that is,
	\begin{align}
	\| S^{\frac{1}{2}}\ep^{k+1}\|^2-((L_1\A_3 +L_2\A_4)\ep^{k},\ep^{k+1})&
	= \|(\A_1 +\A_2)^{\frac{1}{2}}\ep^{k+1}\|^2+\frac{L_1}{2}(\|\A_3^{\frac{1}{2}}\ep^{k+1}\|^2-\|\A_3^{\frac{1}{2}}\ep^{k}\|^2
	+\| \A_3^{\frac{1}{2}}(\ep^{k+1}-\ep^k)\|^2) \nonumber\\&~+\frac{L_2}{2}\left( \|\A_4^{\frac{1}{2}}\ep^{k+1}\|^2
	-\|\A_4^{\frac{1}{2}}\ep^{k}\|^2+\| \A_4^{\frac{1}{2}}(\ep^{k+1}-\ep^k)\|^2\right)\nonumber\\
	&
	\geq \|(\A_1 +\A_2)^{\frac{1}{2}}\ep^{k+1}\|^2 
	+\frac{L_1}{2}\|\A_3^{\frac{1}{2}}\ep^{k+1}\|^2+\frac{L_2}{2}\|\A_4^{\frac{1}{2}}\ep^{k+1}\|^2
	 \nonumber\\&
	 -\frac{L_1}{2}\|\A_3^{\frac{1}{2}}\ep^{k}\|^2-\frac{L_2}{2}\|\A_4^{\frac{1}{2}}\ep^{k}\|^2
	 +\left(\frac{L_2}{2}+\frac{L_1R\lambda_0}{2n}\right)\| \A_4^{\frac{1}{2}}(\ep^{k+1}-\ep^k)\|^2.
	 \label{3kplus}
	\end{align}
Second, we estimate $ (\Divu\eu^{k+1}-\Divu\eu^{k},\ep^{k+1})$. By setting $\bw=\eu^{k+1}-\eu^k$ in \eqref{uzawaq1a} we obtain
	\begin{align}
	(\ep^{k+1},\Divu (\eu^{k+1}-\eu^{k}))&=(\eps(\eu^{k+1}-\eu^{k}), \eps(\eu^{k+1}))+ \lambda  (\divv(\eu^{k+1}-\eu^{k}),\divv \eu^{k+1})\nonumber \\
	&\leq\frac{1}{2}(\|\eps(\eu^{k+1}-\eu^{k})\|^2 + \lambda \|\divv(\eu^{k+1}-\eu^{k} )\|^2)+\frac{1}{2}(\|\eps(\eu^{k+1})\|^2 + \lambda \|\divv\eu^{k+1}\|^2).\label{3ddplus}
	\end{align}
	In order to estimate the right-hand side of~\eqref{3ddplus},
we subtract the $k$-th error from the $(k+1)$-th error and choose $\bw=\eu^{k+1}-\eu^{k}$ in~\eqref{uzawaq1a} and herewith obtaining  
	\begin{align*}
	\|\eps(\eu^{k+1}-\eu^{k})\|^2+ \lambda \|\divv(\eu^{k+1}-\eu^{k})\|^2
	=(\sum_{i=1}^{n}(\be^{k+1}_{p_i}-\be^{k}_{p_i}),\divv (\eu^{k+1}-\eu^{k})).
	\end{align*}
	Applying Cauchy's inequality further yields 
	\begin{align}
	\|\eps(\eu^{k+1}-\eu^{k})\|^2+ \lambda \|\divv(\eu^{k+1}-\eu^{k})\|^2&=(\sum_{i=1}^{n}(\be^{k+1}_{p_i}-\be^{k}_{p_i}),\divv (\eu^{k+1}-\eu^{k}))\leq
	\|\sum_{i=1}^{n}(\be^{k+1}_{p_i}-\be^{k}_{p_i})\|\cdot\|\divv (\eu^{k+1}-\eu^{k})\|\nonumber\\&=\sqrt{\lambda_0}\|\A_4^{\frac{1}{2}}(\ep^{k+1}-\ep^{k})\|\cdot\|\divv (\eu^{k+1}-\eu^{k})\|.\label{cauchy}
	\end{align}
	Noting that
	\begin{align}
	(c_K^2 + \lambda )\|\divv \bw\|^2 \leq\|\eps(\bw)\|^2+ \lambda \|\divv\bw\|^2, \label{3korn}
	\end{align}
	which follows from~\eqref{ck}, 
	we directly obtain
\begin{align*}
	(c_K^2 + \lambda )\|\divv(\eu^{k+1}-&\eu^{k})\|^2 
	\leq \sqrt{\lambda_0}\|\A_4^{\frac{1}{2}}(\ep^{k+1}-\ep^{k})\|\cdot\|\divv (\eu^{k+1}-\eu^{k})\|,
	\end{align*}
	from~\eqref{cauchy}. The latter estimate implies
	\begin{align*}
	&\|\divv(\eu^{k+1}-\eu^{k})\|\leq
	\frac{\sqrt{\lambda_0}}{c_K^2+\lambda}\|\A_4^{\frac{1}{2}}(\ep^{k+1}-\ep^{k})\|.
	\end{align*}
	By using the above inequality in \eqref{cauchy}, it follows that
	\begin{align}
	\|\eps(\eu^{k+1}-\eu^{k})\|^2+ \lambda \|\divv(\eu^{k+1}-\eu^{k})\|^2 &\leq
	\frac{\lambda_0}{c_K^2+\lambda}\|\A_4^{\frac{1}{2}}(\ep^{k+1}-\ep^{k})\|^2. \label{3dplus}
	\end{align}
	\\
	Now, combining \eqref{3ddplus} and \eqref{3dplus} yields
	\begin{align}
	(\ep^{k+1},\Divu (\eu^{k+1}-\eu^{k})) \leq\frac{\lambda_0}{2(c_K^2+\lambda)}\|\A_4^{\frac{1}{2}}(\ep^{k+1}-\ep^{k})\|^2 +\frac{1}{2}(\|\eps(\eu^{k+1})\|^2 + \lambda \|\divv\eu^{k+1}\|^2).\label{divk+1}
	\end{align}
	Finally, inserting \eqref{3kplus} and \eqref{divk+1} in~\eqref{3ee} we have that
	\begin{align*}
	\|\eps(\eu^{k+1})\|^2+&\lambda\|\divv\eu^{k+1}\|^2
	+\|(\A_1 +\A_2)^{\frac{1}{2}}\ep^{k+1}\|^2+\frac{L_1}{2}\|\A_3^{\frac{1}{2}}\ep^{k+1}\|^2
	+\frac{L_2}{2} \|\A_4^{\frac{1}{2}}\ep^{k+1}\|^2 \\&\leq \frac{\lambda_0}{2(c_K^2+\lambda)}\|\A_4^{\frac{1}{2}}
	(\ep^{k+1}-\ep^{k})\|^2
	 +\frac{1}{2}(\|\eps(\eu^{k+1})\|^2 + \lambda \|\divv\eu^{k+1}\|^2) 
	 - \|\R_v^{\frac{1}{2}}\ev^{k+1}\|^2 \\&~~+ \frac{L_1}{2}\|\A_3^{\frac{1}{2}}\ep^{k}\|^2
	 +\frac{L_2}{2}\|\A_4^{\frac{1}{2}}\ep^{k}\|^2 
	 -\left(\frac{L_2}{2}+\frac{L_1R\lambda_0}{2n}\right)\| \A_4^{\frac{1}{2}}(\ep^{k+1}-\ep^k)\|^2
	\end{align*}
which shows \eqref{3last}.
\end{proof}

The next lemma provides a preliminary estimate for the pressure errors.

\begin{lemma}\label{lemma_errors_p}
	Consider Algorithm~\ref{alg1} with $S$ as in~\eqref{defS}. Then the errors $ \be_{\bp}^k$ defined in \eqref{error} satisfy
	\begin{align*}
	&\frac{\lambda_0}{2(\beta_s^{-2}+\lambda)}\|\A_4^{\frac{1}{2}}\ep^{k+1}\|^2+
	\text{$\beta_d^2\|\A_3^{\frac{1}{2}}\ep^{k+1}\|^2$}+ \|(\A_1 +\A_2)^{\frac{1}{2}}\ep^{k+1}\|^2+\frac{L_1}{2}\|\A_3^{\frac{1}{2}}\ep^{k+1}\|^2+ \frac{L_2}{2} \|\A_4^{\frac{1}{2}}\ep^{k+1}\|^2\nonumber\\
	&\qquad \leq \frac{L_1}{2}\|\A_3^{\frac{1}{2}}\ep^k\|^2+\frac{L_2}{2}\|\A_4^{\frac{1}{2}}\ep^{k}\|^2+\left(\frac{\lambda_0}{2(c_K^2+\lambda)}-\frac{L_2}{2}-
	\frac{L_1R\lambda_0}{2n}\right)\|\A_4^{\frac{1}{2}}(\ep^{k+1}-\ep^{k})\|^2.
	\end{align*}\label{alg3le2}
\end{lemma}
\begin{proof}
From Lemma~\ref{lemma1}, it follows that for all $\sum_{i=1}^{n}\be_{p_i}^{k+1} \in P_i$ there exists 
$\bw_0 \in \bU$ such that $\divv\bw_0 = \frac{1}{\sqrt{\lambda_0}}\sum_{i=1}^{n}\be_{p_i}^{k+1}$ and 
$ \|\bw_0\|_{1} \leq  \beta_{s}^{-1}\frac{1}{\sqrt{\lambda_0}}\|\sum_{i=1}^{n}\be_{p_i}^{k+1} \|
=\beta_{s}^{-1}\|\A_4^{\frac{1}{2}}\ep^{k+1}\|$. 
Also,
\begin{align}
\Divu\bw_0=\left( \begin{array}{c}\frac{1}{\sqrt{\lambda_0}}\sum_{i=1}^{n}\be_{p_i}^{k+1}\\\vdots\\\frac{1}{\sqrt{\lambda_0}}\sum_{i=1}^{n}\be_{p_i}^{k+1}\end{array}\right)=\sqrt{\lambda_0}\A_4\ep^{k+1}.\nonumber
\end{align}
Setting $\bw=\bw_0$ in \eqref{uzawaq1a}, it follows that
\begin{align*}
\sqrt{\lambda_0}\|\A_4^{\frac{1}{2}}\ep^{k+1}\|^2&=(\eps(\eu^{k+1}), \eps(\bw_0))+ \lambda  (\divv\eu^{k+1},\divv \bw_0)\leq (\|\eps(\eu^{k+1})\|^2+\lambda \|\divv\eu^{k+1}\|^2)^\frac{1}{2}\cdot(\| \eps(\bw_0)\|^2+\lambda\|\divv \bw_0\|^2)^\frac{1}{2}\\ 
&\leq  (\|\eps(\eu^{k+1})\|^2+\lambda \|\divv\eu^{k+1}\|^2)^\frac{1}{2}\cdot(\beta_{s}^{-2}\|\A_4^{\frac{1}{2}}\ep^{k+1}\|^2+\lambda\|\A_4^{\frac{1}{2}}\ep^{k+1}\|^2)^\frac{1}{2}\\&=(\|\eps(\eu^{k+1})\|^2+\lambda \|\divv\eu^{k+1}\|^2)^\frac{1}{2}\cdot(\beta_{s}^{-2}+\lambda)^\frac{1}{2} \|\A_4^{\frac{1}{2}}\ep^{k+1}\|
\end{align*}
and, therefore,
\begin{align}
\frac{\lambda_0}{\beta_{s}^{-2}+\lambda}\|\A_4^{\frac{1}{2}}\ep^{k+1}\|^2 \leq
	\|\eps(\eu^{k+1})\|^2 + \lambda \|\divv\eu^{k+1}\|^2.
	\label{3u4}
\end{align}
Using \eqref{3u4} and \eqref{u3} in \eqref{3last}, we have
\begin{align}
&\frac{\lambda_0}{2(\beta_s^{-2}+\lambda)}\|\A_4^{\frac{1}{2}}\ep^{k+1}\|^2+
\text{$\beta_d^2\|\A_3^{\frac{1}{2}}\ep^{k+1}\|^2$}+\|(\A_1 +\A_2)^{\frac{1}{2}}\ep^{k+1}\|^2
+\frac{L_1}{2}\|\A_3^{\frac{1}{2}}\ep^{k+1}\|^2+ \frac{L_2}{2} \|\A_4^{\frac{1}{2}}\ep^{k+1}\|^2\nonumber\\
&\leq\frac{L_1}{2}\|\A_3^{\frac{1}{2}}\ep^k\|^2+\frac{L_2}{2}\|\A_4^{\frac{1}{2}}\ep^{k}\|^2
+\left(\frac{\lambda_0}{2(c_K^2+\lambda)}-\frac{L_2}{2}-\frac{L_1R\lambda_0}{2n}\right)
\|\A_4^{\frac{1}{2}}(\ep^{k+1}-\ep^{k})\|^2.\label{3last1}
\end{align}
\end{proof}
The following two theorems present the main convergence results for Algorithm~\ref{alg1}.
\begin{theorem}\label{Theorem3} 
	Consider Algorithm~\ref{alg1}. For any $\theta>0$ and 
	$L_2\geq \frac{\lambda_0}{(c_K^2+\lambda)\left(1+\frac{\theta\beta_d^2R\lambda_0}{n}\right)}$,  
	$L_1=\theta\beta_d^2L_2$, the errors $ \be_{\bp}^k$ defined in \eqref{error} satisfy the estimate:
	\begin{equation}\label{theoep}
	\|\ep^{k+1}\|^2_{\bP_{\theta}} \leq \mathrm{rate}^2(\lambda, R,\theta) \|\ep^{k}\|^2_{\bP_{\theta}}	
	\end{equation}
	with 
	\begin{equation*}\label{constant_P}
	\mathrm{rate}^2(\lambda, R,\theta)\leq \frac{1}{C+1},\, 
	C=\min \left\{ \frac{\lambda_0}{\beta_s^{-2}+\lambda}, 2\theta^{-1}\right\}L_2^{-1} 
	\end{equation*}
	and
	\begin{equation}\label{norm_Ptheta}
	\|\ep^{k+1}\|^2_{\bP_{\theta}}
	 :=\|\A_4^{\frac{1}{2}}\ep^{k+1}\|^2+\theta\beta_d^2\text{$\|\A_3^{\frac{1}{2}}\ep^{k+1}\|^2$}+ \|(\A_1 +\A_2)^{\frac{1}{2}}\ep^{k+1}\|^2.
	 \end{equation}
	\begin{enumerate}
	\item For $\theta=\theta_0:=\beta_d^{-2}$ and $L_2= \frac{\lambda_0}{(c_K^2+\lambda)\left(1+\frac{R\lambda_0}{n}\right)}$, 
	we obtain the convergence factor under the norm 
	$\|\cdot\|_{\bm P_{\theta_0}}$ estimated by 
	$$
	\mathrm{rate}^2(\lambda, R)\leq \frac{1}{\frac{c_0(c_K^2+\lambda)\left(1+\frac{\lambda_0R}{n}\right)}{\lambda_0}+1} \leq \max\left\{\frac{1}{c_0+1},\frac{1}{c_0c_K^2+1},\frac{1}{2}\right\},
	\hbox{where}~c_0=\min \left\{ \frac{\lambda_0}{\beta_s^{-2}+\lambda}, 2\beta_d^2 \right\}.
	$$\label{alg3th0}
	Here for any $\bx\in \bP$
	$$
	\|\bx\|_{\bm P_{\theta_0}}^2 : = \|\A_3^{\frac{1}{2}}\bx\|^2
	+\text{$\|\A_4^{\frac{1}{2}}\bx \|^2$}+ \|(\A_1 +\A_2)^{\frac{1}{2}}\bx\|^2.
	$$
\item For the best choice 
	$\theta=\theta_*:=\frac{2(\beta_s^{-2}+\lambda)}{\lambda_0}$
	and 
	$L_2= \frac{\lambda_0}{(c_K^2+\lambda)\left(1+\frac{2\beta_d^2(\beta_s^{-2}+\lambda)R}{n}\right)}$, 
	the errors $ \be_{\bp}^k$ satisfy the estimate 
	\begin{align}\label{alg3th1}
	\|\ep^{k+1}\|^2_{\bP_{\theta_*}} \le
	\mathrm{rate}^2(\lambda, R)\leq \frac{1}{\frac{(c_K^2+\lambda)
	\left(1+\frac{2\beta_d^2(\beta_s^{-2}+\lambda)R}{n}\right)}{(\beta_s^{-2}+\lambda)}+1}\leq 
	\max\left\{\frac{\beta_s^{-2}}{c_K^2+\beta^{-2}_s},\frac{1}{2}\right\},
	\end{align}
	where 
	\begin{equation}\label{norm_Pv}
	\|\ep^{k+1}\|^2_{\bP_{\theta_\ast}}
	:= \|\A_4^{\frac{1}{2}}\ep^{k+1}\|^2+\frac{2(\beta_s^{-2}+\lambda)\beta_v^2}{\lambda_0}
	\text{$\|\A_3^{\frac{1}{2}}\ep^{k+1}\|^2$}+ \|(\A_1 +\A_2)^{\frac{1}{2}}\ep^{k+1}\|^2.
	 \end{equation}
	\end{enumerate}
\end{theorem}
\begin{proof}
In view of the estimate presented in Lemma~\ref{lemma_errors_p}, we want to find $L_2$ and $L_1$ subject to the condition 
\begin{equation}\label{condition}
\frac{\lambda_0}{2(c_K^2+\lambda)}-\frac{L_2}{2}-\frac{L_1R\lambda_0}{2n}\leq 0.
\end{equation}
For any $\theta>0$, we rewrite~\eqref{3last1} as 
\begin{align}
&\frac{\lambda_0}{2(\beta_s^{-2}+\lambda)}\|\A_4^{\frac{1}{2}}\ep^{k+1}\|^2+
\theta^{-1}\theta\text{$\beta_d^2\|\A_3^{\frac{1}{2}}\ep^{k+1}\|^2$}+ \|(\A_1 +\A_2)^{\frac{1}{2}}\ep^{k+1}\|^2
+\frac{L_1}{2\theta\beta_d^2}\theta\beta_d^2\|\A_3^{\frac{1}{2}}\ep^{k+1}\|^2
+ \frac{L_2}{2} \|\A_4^{\frac{1}{2}}\ep^{k+1}\|^2\nonumber
\\ &\qquad \leq \frac{L_1}{2\theta\beta_d^2}\theta\beta_d^2\|\A_3^{\frac{1}{2}}\ep^k\|^2
+\frac{L_2}{2}\|\A_4^{\frac{1}{2}}\ep^{k}\|^2+\left(\frac{\lambda_0}{2(c_K^2+\lambda)}
-\frac{L_2}{2}-\frac{L_1R\lambda_0}{2n}\right)\|\A_4^{\frac{1}{2}}(\ep^{k+1}-\ep^{k})\|^2\label{theta},
\end{align}
namely, 
\begin{align}
&\left(\frac{\lambda_0}{2(\beta_s^{-2}+\lambda)}+\frac{L_2}{2}\right)\|\A_4^{\frac{1}{2}}\ep^{k+1}\|^2+
\left(\theta^{-1}+\frac{L_1}{2\theta\beta_d^2}\right)\theta\text{$\beta_d^2\|\A_3^{\frac{1}{2}}\ep^{k+1}\|^2$}
+ \|(\A_1 +\A_2)^{\frac{1}{2}}\ep^{k+1}\|^2\nonumber
\\ & \qquad\leq\frac{L_1}{2\theta\beta_d^2}\theta\beta_d^2\|\A_3^{\frac{1}{2}}\ep^k\|^2+\frac{L_2}{2}\|\A_4^{\frac{1}{2}}\ep^{k}\|^2.\label{theta1}
\end{align}
Then, for $L_2 \leq 1$ we obtain 
\begin{align}
&\min\left\{\frac{\lambda_0}{2(\beta_s^{-2}+\lambda)}+\frac{L_2}{2}, \theta^{-1}+\frac{L_1}{2\theta\beta_d^2}\right\} \left(\|\A_4^{\frac{1}{2}}\ep^{k+1}\|^2+
\theta\text{$\beta_d^2\|\A_3^{\frac{1}{2}}\ep^{k+1}\|^2$}+ \|(\A_1 +\A_2)^{\frac{1}{2}}\ep^{k+1}\|^2\right)\nonumber\\
&\qquad\leq\max\left\{\frac{L_1}{2\theta\beta_d^2},\frac{L_2}{2}\right\}\left(\theta\beta_d^2\|\A_3^{\frac{1}{2}}\ep^k\|^2+\|\A_4^{\frac{1}{2}}\ep^{k}\|^2\right).\label{theta2}
\end{align}
Now, choose $L_1= \theta\beta_d^2L_2$. Then, condition~\eqref{condition} becomes 
\begin{equation}\label{condition1}
\frac{\lambda_0}{2(c_K^2+\lambda)}-\frac{L_2}{2}-\frac{\theta\beta_d^2L_2R\lambda_0}{2n}\leq 0~~ \text{or}~~  L_2\geq \frac{\frac{\lambda_0}{(c_K^2+\lambda)}}{1+\frac{\theta\beta_d^2R\lambda_0}{n}}
\end{equation}
and we can simplify \eqref{theta2} as follows 
\begin{align*}
&\min\left\{\frac{\lambda_0}{2(\beta_s^{-2}+\lambda)}+\frac{L_2}{2}, \theta^{-1}+\frac{L_2}{2}\right\} \left(\|\A_4^{\frac{1}{2}}\ep^{k+1}\|^2+
\theta\text{$\beta_d^2\|\A_3^{\frac{1}{2}}\ep^{k+1}\|^2$}+ \|(\A_1 +\A_2)^{\frac{1}{2}}\ep^{k+1}\|^2\right)\nonumber\\
&\qquad\leq\frac{L_2}{2}\left(\theta\beta_d^2\|\A_3^{\frac{1}{2}}\ep^k\|^2+\|\A_4^{\frac{1}{2}}\ep^{k}\|^2\right),\label{theta3}
\end{align*}
which shows~\eqref{theoep}.
Statements 1. and 2. are direct consequences of~\eqref{theoep} for the particular choices of $\theta$ in the corresponding norms.

\end{proof}
Note that estimate~\eqref{alg3th1} not only indicates that the convergence rate of the Uzawa-type iterative method has a
uniform, with respect to the parameters, upper-bound being strictly less than $1$, but also that it is bounded by a number much smaller than $1$ if $\lambda$ is large.
Moreover, the presented analysis of the Uzawa-type scheme results in a new, parameter-optimized block-triangular preconditioner that can be
used to accelerate the convergence of the GMRES method if the latter is applied to the augmented system~\eqref{Augweak}.
The parameter-robust uniform convergence estimates for the new Uzawa-type method imply the field-of-values equivalence of this preconditioner
for the augmented system.

\begin{theorem}\label{theorem_errors_vu}
	Consider Algorithm~\ref{alg1} with $S$ as introduced in~\eqref{defS}. Then the errors $\be_{\bu}^k$ and $\be_{\bv}^k$ 
	defined in \eqref{error} satisfy the estimates:
	\begin{equation}
	\|\eu^{k}\|_{\bU} \leq C_u [\mathrm{rate}(\lambda,R)]^k,\quad \|\ev^{k}\|_{\bV_{\theta_*}} \leq C_v [\mathrm{rate}(\lambda,R)]^k
	\end{equation}
	where 
	\begin{equation}\label{norms_vu}
	\|\ev^{k}\|_{\bV_{\theta_*}}^2:=(\R_v\ev^{k},\ev^{k}) + (S^{-1}\mathrm{Div} \ev^{k},\mathrm{Div} \ev^{k}),\quad
	\|\bu\|^2_{\bU}  := \|\eps(\bu)\|^2 + \lambda\|\mathrm{div}\bu\|^2
	\end{equation} 
	and the constants $C_u$ and $C_v$ are independent of 
	the model parameters and the time step size.\label{alg3th2}
\end{theorem}
\begin{proof}
First, we estimate $\|\eu^{k+1}\|^2_{\bU}$. By setting $\bw=\eu^{k+1}$ in \eqref{uzawaq1a}, applying Cauchy's inequality 
and using~\eqref{3korn} 
we obtain
\begin{align*}
\|\eps(\eu^{k+1})\|^2 + \lambda \|\divv\eu^{k+1}\|^2 &=\left(\sum_{i=1}^{n}\be^{k+1}_{p_i},\divv\eu^{k+1}\right) \leq \left\|\sum_{i=1}^{n}\be^{k+1}_{p_i}\right\|\cdot\|\divv\eu^{k+1}\|
=\sqrt{\lambda_0}\|\A_4^{\frac{1}{2}}\ep^{k+1}\| \cdot\|\divv\eu^{k+1}\|\\&\leq\sqrt{\lambda_0}\|\A_4^{\frac{1}{2}}\ep^{k+1}\|\cdot\sqrt{\frac{1}{c_K^2+\lambda}(\|\eps(\eu^{k+1})\|^2 + \lambda \|\divv\eu^{k+1}\|^2)},
\end{align*}
or, equivalently, 
\begin{equation}\label{erroru}
\|\eu^{k+1}\|^2_{\bU} \leq\frac{\lambda_0}{c_K^2+\lambda}\|\A_4^{\frac{1}{2}}\ep^{k+1}\|^2 \leq\frac{\lambda_0}{c_K^2+\lambda}\|\ep^{k+1}\|^2_{\bP_\theta}.
\end{equation}
In order to estimate $\|\ev^{k+1}\|^2_{\bV_{\theta_*}}$ 
we set $\bz=\ev^{k+1}$ in~\eqref{uzawaq1b} and apply the Cauchy inequality to derive
\begin{align}
(\R_v\ev^{k+1},\ev^{k+1}) + (S^{-1}&\Divv \ev^{k+1},\Divv \ev^{k+1}) = (\ep^{k},\Divv \ev^{k+1}) 
- (S^{-1}\Divu \eu^{k},\Divv \ev^{k+1})  - (S^{-1}(\A_1+\A_2)\ep^k,\Divv \ev^{k+1})\nonumber \\
&= (S^{-1}(L_1\A_3 + L_2\A_4)\ep^{k},\Divv \ev^{k+1}) - (S^{-1}\Divu \eu^{k},\Divv \ev^{k+1}) \nonumber \\
&\leq (S^{-1}(L_1\A_3 + L_2\A_4)\ep^{k},(L_1\A_3 + L_2\A_4)\ep^{k})\nonumber\\&
+\frac{1}{4}(S^{-1}\Divv \ev^{k+1},\Divv \ev^{k+1})+ (S^{-1}\Divu \eu^{k},\Divu \eu^{k}) 
+\frac{1}{4}(S^{-1}\Divv \ev^{k+1},\Divv \ev^{k+1}). \label{estimatev}
\end{align}
From the definition of $S$, see~\eqref{defS}, that of $\|\cdot\|_{\bP_{\theta_*}}$, see~\eqref{norm_Pv}, 
and noting that $L_1 =\theta\beta_d^2L_2$, see Theorem 6, 
we have  
\begin{align}
(S^{-1}(L_1\A_3 + L_2\A_4)\ep^{k},(L_1\A_3 + L_2\A_4)\ep^{k})& \leq  ((L_1\A_3 + L_2\A_4)^{-1}(L_1\A_3 + L_2\A_4)\ep^{k},(L_1\A_3 + L_2\A_4)\ep^{k})\nonumber\\
&~~=((L_1\A_3 + L_2\A_4)\ep^{k},\ep^{k})\leq L_2\|\ep^{k}\|^2_{\bP_{\theta_*}}.
\end{align}
Then~\eqref{estimatev} can be rewritten in the form
\begin{equation}
(\R_v\ev^{k+1},\ev^{k+1}) + \frac{1}{2}(S^{-1}\Divv \ev^{k+1},\Divv \ev^{k+1}) 
\leq L_2\|\ep^{k}\|^2_{\bP_{\theta_*}}+\| S^{-\frac{1}{2}}\Divu \eu^{k}\|^2.\label{3ev}
\end{equation}
Again, from the definition of $S$, and observing that 
$L_1\A_3 + L_2\A_4=(L_1R I_{n\times n}+\frac{L_2}{\lambda_0}\boldsymbol e\boldsymbol e^T)$, then by choosing $a=L_1R$ and $b=\frac{L_2}{\lambda_0}$ in Lemma~\ref{etauu:7}, it follows that 
\begin{align}
\|S^{-\frac{1}{2}}\Divu \eu^{k}\|^2 &\leq ((L_1\A_3 + L_2\A_4)^{-1} \Divu \eu^{k},\Divu \eu^{k})
=\big((\sum_{\substack{i=1}}^n\sum_{\substack{j=1}}^n b_{ij})\divv \eu^{k},\divv \eu^{k}\big)\\
&=\frac{n\lambda_0}{L_1R\lambda_0+nL_2} (\divv \eu^{k},\divv \eu^{k})\leq  (c_K^2+\lambda) (\divv \eu^{k},\divv \eu^{k}).
\end{align}
Therefore, from~\eqref{erroru}, we have 
\begin{align*}
\|\ev^{k}\|_{\bV_{\theta_*}}^2=(\R_v\ev^{k+1},\ev^{k+1}) + \frac{1}{2}(S^{-1}\Divv \ev^{k+1},\Divv \ev^{k+1}) 
&\leq L_2\|\ep^{k}\|^2_{\bP_{\theta_*}}+(c_K^2+\lambda) (\divv \eu^{k},\divv \eu^{k})\leq L_2\|\ep^{k}\|^2_{\bP_{\theta_*}}+\|\eu^k\|^2_U\\
&\leq L_2\|\ep^{k}\|^2_{\bP_{\theta_*}}+\frac{\lambda_0}{c_K^2+\lambda} \|\ep^{k}\|^2_{\bP_{\theta_*}}
=\left(L_2+\frac{\lambda_0}{c_K^2+\lambda}\right) \|\ep^{k}\|^2_{\bP_{\theta_*}},
\end{align*}
which completes the proof.
\end{proof}

\begin{remark}
Note that for the particular choice of $S$ and $M$ 
in this section, the block triangular matrix 
on the left-hand side of~\eqref{GSsystem} provides a field of values equivalent preconditioner with equivalence constants independent of any 
model and discretization parameters.
\end{remark}

\section{The discrete MPET problem}\label{sec:uni_stab_disc_model} 
Mass conservative discretizations for the MPET model are considered in this section,
cf.~\cite{HongKraus2017parameter,Hong2018conservativeMPET}. 
The analysis here can also be utilized for other stable discretizations of the three-field formulation
of the MPET model, e.g.~\cite{Hu2017nonconforming, rodrigo2018new}.

\subsection{Notation}

Let $\mathcal{T}_h$ be a shape-regular triangulation of the domain $\Omega$ into triangles/tetrahedrons
where the subscript $h$ denotes the mesh-size. 
Furthermore, let $\mathcal{E}_h^{I}$ and $\mathcal{E}_h^{B}$ define the set of all interior edges/faces and the set of all boundary edges/faces of 
$\mathcal{T}_h$ respectively with their union being written as $\mathcal{E}_h$. 

We introduce the following broken Sobolev spaces 
$$
H^s(\mathcal{T}_h)=\{\phi\in L^2(\Omega), \mbox{ such that } \phi|_T\in H^s(T) \mbox{ for all } T\in \mathcal{T}_h \}
$$ 
for $s\geq 1$.

Define $T_1$ and $T_2$ to be two elements from the triangulation which share an edge or face $e$ and $\bm n_1$ and $\bm n_2$ 
to be the corresponding unit normal vectors to $e$ 
which point to the exterior of $T_1$ and $T_2$. 
For $q\in H^1(\mathcal{T}_h)$, $\bm v \in H^1(\mathcal{T}_h)^d$ and $\bm \tau \in H^1(\mathcal{T}_h)^{d\times d}$ 
and any $e\in \mathcal{E}_h^{I}$, the jumps $[\cdot]$ and averages $\{\cdot \}$ are defined by
\begin{equation*}
[q]=q|_{\partial T_1\cap e}-q|_{\partial T_2\cap e},\quad
[\bm v]=\bm v|_{\partial T_1\cap e}-\bm v|_{\partial T_2\cap e}
\end{equation*}
and
\begin{equation*}
\begin{split}
\{\bm v\} &=\frac{1}{2}(\bm v|_{\partial T_1\cap e}\cdot \bm n_1-\bm
v|_{\partial T_2\cap e}\cdot \bm n_2), \quad 
\{\bm \tau\}=\frac{1}{2}(\bm \tau|_{\partial T_1\cap
e} \bm n_1-\bm \tau|_{\partial T_2\cap e} \bm n_2),
\end{split}
\end{equation*}
whereas in the case of  
$e \in  \mathcal{E}_h^{B}$
\[[q]=q|_{e}, ~~ [\bm v]=\bm v|_{e},\quad
\{\bm v\}=\bm v |_{e}\cdot \bm n,\quad
\{\bm \tau\}=\bm \tau|_{e}\bm n.
\]

\subsection{Mixed finite element spaces and discrete formulation}\label{subsec:DGdiscretization}
So as to discretize the flow equations, a mixed finite element method has been used to approximate fluxes and pressures. 
The displacement field of the mechanics problem is approximated using a discontinuous Galerkin method.
%
The following finite element spaces are employed:
\begin{eqnarray*}
\bm U_h&=&\{\bm u \in H(\divv ;\Omega):\bm u|_T \in \bm U(T),~T \in \mathcal{T}_h;~ \bm u \cdot
\bm n=0~\hbox{on}~\partial \Omega\},
\\[1ex]
\bm V_{i,h}&=&\{\bm v \in H(\divv ;\Omega):\bm v|_T \in \bm V_i(T),~T \in \mathcal{T}_h;~ \bm v \cdot
\bm n=0~\hbox{on}~\partial \Omega\},~~i=1,\dots,n,
\\
P_{i,h}&=&\left\{q \in L^2(\Omega):q|_T \in Q_i(T),~T \in \mathcal{T}_h; ~\int_{\Omega}q dx=0\right\},~~i=1,\dots,n,
\end{eqnarray*}
where $\bm V_{i}(T)/Q_{i}(T)={\rm RT}_{l-1}(T)/{\rm P}_{l-1}(T)$, $\bm U(T) = {\rm BDM}_l(T)$ or $\bm U(T) = {\rm BDFM}_l(T)$ for $l\ge 1$. 
One should note that $\divv  \, \bm U(T)=\divv  \, \bm V_i(T)=Q_i(T)$ for each of these choices.

Also, it has been shown in~\cite{HongKraus2017parameter,Hong2018conservativeMPET} that for all $\bm u\in \bm U_h$, 
$[\bm u_n]=0$, from which follows that $[\bm u]=[\bm u_t]$.
%
Here, $\bm u_n$ and $\bm u_t$ are the normal and tangential component of $\bm u$ respectively.

Defining
$$
\bm v_h^T=(\boldsymbol v^T_{1,h}, \cdots \boldsymbol v^T_{n,h}),\qquad \bm p_h^T=(p_{1,h},\cdots, p_{n,h}),$$ 
$$\bm z_h^T=(\boldsymbol z^T_{1,h}, \cdots \boldsymbol z^T_{n,h}),\qquad \bm q_h^T=(q_{1,h},\cdots, q_{n,h}),$$
$$
\bm V_h=\boldsymbol V_{1,h}\times\cdots\times\boldsymbol V_{n,h},
\quad \bm P_h= P_{1,h}\times\cdots\times P_{n,h}, \quad  \bm X_h = \bm U_h\times\bm V_h\times \bm P_h,
$$
then the following discretization of the continuous variational problem results from the weak formulation
of~\eqref{MPET_strong}:
Find $(\boldsymbol u_h;\bm v_h;\bm p_h )\in  \bm X_h$, such that for
any $(\boldsymbol w_h; \bm z_{h};\bm q_{h})\in  \bm X_{h}$
and $i=1,\dots, n$
\begin{subequations}\label{eq:75-77}
\begin{eqnarray}
  (R^{-1}_i\bv_{i,h},\bz_{i,h}) {-} (p_{i,h},\divv \bz_{i,h}) &=& 0,  \label{eq:76} \\
 	\hspace{5ex} -(\divv\bu_h,q_{i,h}) - (\divv\bv_{i,h},q_{i.h})  +\tilde{\alpha}_{ii} (p_{i,h},q_{i,h})
 	 +\sum_{\substack{{j=1}\\j\neq i}}^{n}\alpha_{ij} (p_{j,h},q_{i,h})&=& (g_i,q_{i,h}) ,
	 \label{eq:77}\\
	  a_h(\bm u_h,\bm w_h) +\lambda ( \divv \boldsymbol  u_h, \divv \boldsymbol  w_h)
 - \sum_{i=1}^n(p_{i,h}, \divv \boldsymbol  w_h)&=&(\boldsymbol f, \boldsymbol w_h),\label{eq:75}
\end{eqnarray}
\end{subequations}
where
\begin{eqnarray}
a_h(\bm u,\bm w)&=&\label{78}
\sum _{T \in \mathcal{T}_h} \int_T \ep(\bm{u}) :
\ep(\bm{w}) dx-\sum_{e \in \mathcal{E}_h} \int_e \{\ep(\bm{u})\} \cdot [\bm w_t] ds\\
&&\nonumber-\sum _{e \in \mathcal{E}_h} \int_e \{\ep(\bm{w})\} \cdot [\bm u_t]ds+\sum _{e
\in \mathcal{E}_h} \int_e \eta h_e^{-1}[ \bm u_t] \cdot [\bm w_t] ds,
\end{eqnarray}
$\tilde{\alpha}_{ii}=-\alpha_{p_i}-\alpha_{ii}$, and $\eta $ is a stabilization parameter which is independent of 
$\lambda,\,R_i^{-1}$, $\alpha_{p_i}$,  ${\alpha}_{ij}$, $i,j \in \{1,\dots,n\}$, the network scale $n$,
and the mesh size $h$.

In the derivation of the discrete variational problem~\eqref{eq:75-77}, homogeneous Dirichlet boundary conditions for $\bm u$
and homogeneous Neumann boundary conditions for $p_i$, $i=1,2,\ldots,n$, have been assumed for each case over the entire domain boundary.
The DG discretizations for more general (rescaled) boundary conditions and the stability analysis of the related discrete variational problems can
be found in~\cite{Hong2018conservativeMPET, HongKraus2017parameter}.
The iterative scheme for flux-pressure-displacement formulation of the discrete MPET problem, analogous to Algorithm \ref{alg1}, is as follows:
\begin{algorithm}[H]
	\caption{
	Fully decoupled iterative scheme for flux-pressure-displacement formulation of discrete MPET problem}
	\begin{algorithmic}[1]
 \Statex{\textbf{Step a:} Given $\bp_h^k$ and $\bu_h^k$, we first solve for $\bv_h^{k+1}$, such that for all ${\boldsymbol z}_h \in {\boldsymbol V}_h$ there holds}
\begin{align}
&(A_v\bv_h^{k+1},\bz_h)  +(\M\Divv \bv_h^{k+1},\Divv \bz_h) =- (\M\bg,\Divv \bz_h) + ((I- \M(\A_1+\A_2))\bp_h^k,\Divv \bz_h)-  (\M\Divu \bu_h^k,\Divv \bz_h).
\end{align}
 \Statex{\textbf{Step b:} Given $\bu_h^k$ and $\bv_h^{k+1}$, we solve for $\bp_h^{k+1}$, such that for all ${\boldsymbol q}_h \in {\boldsymbol P}_h$ there holds}
 \begin{align}
(\S\bp_h^{k+1},\bq_h)=-(\bg,\bq_h)+(\S\bp_h^k,\bq_h)- ((\A_1 +\A_2)\bp_h^k,\bq_h))-(\Divu\bu_h^k,\bq_h) - (\Divv\bv_h^{k+1},\bq_h) .
 \end{align}
 \Statex{\textbf{Step c:} Given $\bp_h^{k+1}$ and $\bv_h^{k+1}$, we solve for $\bu_h^{k+1}$, such that for all ${\boldsymbol w}_h \in {\boldsymbol U}_h$ there holds}
 \begin{align}
&a_h(\bu_h^{k+1},\bm w_h) = (\bf,\bw_h) + (\bp_h^{k+1},\Divu \bw_h).
 \end{align}
\end{algorithmic}
\label{alg2}
\end{algorithm}
In \textbf{Step a}, a coupled $H(\rm div)$ problem is solved. 
As mentioned in Remark 6 of~\cite{Hong2018conservativeMPET}, we can apply orthogonal transformations to the flux and pressure
subsystems 
which decouple the fluxes from each other and also the pressures from each other. 
For fluxes,
this procedure
results in $n$ decoupled $H(\rm div)$ problems for the operators $I+\bar{\mu}_i\nabla {\rm div}$, $i=1,2, \cdots, n$, where $\bar{\mu}_i$
are the eigenvalues of an $n{\times}n$ coefficient matrix, $n$ denoting the number of networks, i.e., $n \in \{ 1,2,4,8\}$ in the examples
presented in the next section; correspondingly, the decoupling of the pressure subsystem yields~$n$, essentially, well conditioned independent
$L^2$ problems.

There are several works addressing the solution of nearly singular $H(\rm div)$ problems and we may resort to Hiptmair-Xu preconditioners~\cite{Hiptmair.R;Xu.J2007a}
and the robust subspace correction methods \cite{xu1992iterative, xu2002method, lee2007robust, lee2008sharp}.
There also exist multigrid methods that serve this purpose, see, e.g.~\cite{vassilevski1996preconditioning,arnold1997preconditioning}. In case of highly
varying permeability (conductivity) coefficient, the auxiliary space multigrid preconditioners based on additive Schur complement approximation proposed
by Kraus et al.~\cite{kraus2016preconditioning} provide a prameter-robust alternative.

In~\textbf{Step c}, to obtain $A_u^{-1}$ for the elasticity subproblem, one can use the multigrid method 
proposed in \cite{hong2016robust} for the discontinuous
Galerkin discretization and the multigrid methods proposed in \cite{schoberl1999multigrid, lee2009robust} for conforming elements, which are all robust with
respect to the Lam\'{e} parameter $\lambda$.
Following the methodology of the convergence analysis presented for the continuous MPET problem, statements analogous to those presented
in Theorem~\ref{Theorem3} and Theorem~\ref{theorem_errors_vu} can also be proven for Algorithm~\ref{alg2}.

\section{Numerical results}\label{sec:Section_4}

In the following, we consider four numerical test settings to demonstrate the effectiveness and accuracy of the proposed 
Uzawa-type iterative schemes for the MPET model.

First, numerical results are presented for the single network problem, i.e., the Biot model, in Figure~\ref{figure_Biot}.
These validate the theoretical convergence estimates of the linear stationary iterative method based on
Algorithm~\ref{alg1} which has been additionally assessed against the preconditioned GMRES algorithm. 
In the second and third tests, 
the performance of Algorithm~\ref{alg1} is compared with the preconditioned GMRES algorithm 
and the fixed-stress algorithm as proposed in~\cite{Hong2018parameter-robust}, 
cf.~\eqref{fixed_stress_sys}, 
for the two-network and four-network MPET problems. Finally, a scaling test demonstrating the behaviour 
of the preconditioned GMRES and the Uzawa-type algorithms for different numbers of networks is performed.

The block Gauss-Seidel preconditioner that we have used to accelerate the GMRES method 
equals the lower block triangular matrix in the left-hand side of~\eqref{GSsystem} where $M=S^{-1}$ and $S$ is given in~\eqref{defS}. 

All the numerical results in this section have been conducted on the FEniCS computing platform, see e.g.~\cite{AlnaesBlechta2015a,LoggMardalEtAl2012a}. 
In all test cases the set-up is as follows:
\begin{itemize} 
\item The domain $\Omega\in \mathbb{R}^2$ is the unit square which is partioned into $2N^2$ congruent right-angled triangles; 
\vspace{1ex}
\item 
The discretization setting is the same as in~\cite{Hong2018parameter-robust,Hong2018conservativeMPET}, i.e.,
we use discontinuous piecewise constant elements, lowest-order Raviart-Thomas elements and 
Brezzi-Douglas-Marini elements
to approximate the pressures, fluxes and displacement fields respectively;
\vspace{1ex}
\item For all experiments conducted using Algorithm~\ref{alg1} we set 
$$L_2=\frac{\lambda_0}{(c_k^2+\lambda)(1+\frac{2\beta_d^2(\beta_s^{-2}+\lambda)R}{n})}, \quad
L_1=\frac{2(\beta_s^{-2}+\lambda)\beta_d^2}{\lambda_0}L_2$$ 
and 
$\beta_s^2=\beta_d^2=0.18$, see~\cite{costabel2015ontheinequalities}.
\vspace{1ex}
\item The stopping criterium of the iterative process is the reduction of the initial preconditioned residual 
by a factor~$10^8$ where a random vector has been used in the initialization. 
%
\end{itemize}

\subsection{The Biot's consolidation model}

Consider system~\eqref{eq:MPET} for $n=1$, i.e., a system for which only one pressure and one flux exists, where for $(x,y)\in \Omega$
$$
g=R_1\left(\frac{\partial \phi_2}{\partial x}+\frac{\partial \phi_2}{\partial y}\right)-\alpha_{p_1}(\phi_2-1),
$$  
$$
\phi_1=(x-1)^2(y-1)^2 x^2 y^2, \qquad
\phi_2=900(x-1)^2(y-1)^2 x^2 y^2
$$
and
$$
\bf=\left(
\begin{array}{c}
-(2y^3-3y^2+y)(12x^2-12x+2)-(x-1)^2x^2(12y-6)+900(y-1)^2y^2(4x^3-6x^2+2x)\\
~~(2x^3-3x^2+x)(12y^2-12y+2)+(y-1)^2y^2(12x-6)+900(x-1)^2x^2(4y^3-6y^2+2y)
\end{array}
\right).
$$

Experiments over a wide-range of input parameters $\alpha_p$, $\lambda$, $R_1^{-1}$ have been run with 
Algorithm~\ref{alg1} and the preconditioned GMRES algorithm and are shown in Figure~\ref{figure_Biot}. 
In all test cases, the number of Uzawa-type iterations required to achieve the prescribed solution accuracy is
bounded by a constant 
independent of all model and discretization parameters. Clearly, the GMRES preconditioned algorithm demonstrates 
better convergence behaviour for small $\lambda$. 


\begin{figure}[h!]
\includegraphics[scale=0.5]{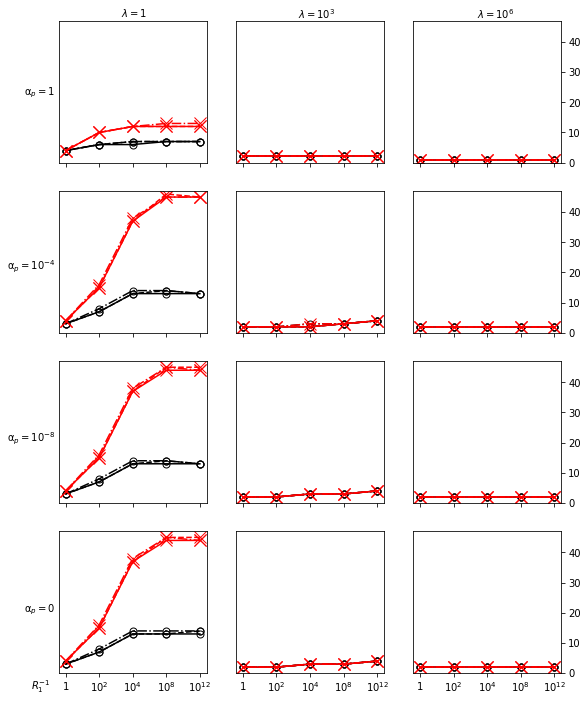}
\caption{Number of preconditioned GMRES (small black circles) and augmented Uzawa-type (red crosses) iterations 
for preconditioned residual reduction by a factor $10^8$ when solving the Biot problem. 
These tests have been performed for $h=1/32$ (dash-dotted line), $h=1/64$ (dashed line) and $h=1/128$ (full line).}
\label{figure_Biot}
\end{figure}

\subsection{The Biot-Barenblatt model}

In the next test, system~\eqref{eq:MPET} is considered for $n=2$ where the problem setting is as per the cantilever bracket benchmark problem 
in~\cite{NAFEMS1990}. 
We denote the bottom, right, top and left parts of $\Gamma=\partial \Omega$ 
by $\Gamma_1$, $\Gamma_2$, $\Gamma_3$ and $\Gamma_4$
and, also, we impose $\boldsymbol u=0$ on $\Gamma_4$,
$(\boldsymbol {\sigma} -p_1\boldsymbol I-p_2\boldsymbol I)\boldsymbol n=(0,0)^T$ on $\Gamma_1\cup\Gamma_2$, 
$(\boldsymbol {\sigma} -p_1\boldsymbol I-p_2\boldsymbol I)\boldsymbol n=(0,-1)^T$ on $\Gamma_3$, $p_1=2$ on
$\Gamma$ and $p_2=20$ on $\Gamma$. 
Further, we set $\boldsymbol f=\boldsymbol 0$, $g_1=0$ and $g_2=0$. 
Table~\ref{parameters_Barenblatt} 
shows the reference values of the model parameters as given in~\cite{Kolesov2017}. 

\begin{table}
\caption{Reference values of model parameters for the Barenblatt model.}
\label{parameters_Barenblatt}
\begin{center}
\begin{tabular}{ccc}
\hline \hline
parameter & value & unit \\ \hline \hline
$\widehat{\lambda}$ & $4.2 *10^6$ & Nm$^{-2}$ \\
$\mu$ & $2.4*10^6$ & Nm$^{-2}$ \\
$c_{p_1}$ & $5.4*10^{-8}$ & N$^{-1}$m$^2$ \\
$c_{p_2}$ & $1.4*10^{-8}$ & N$^{-1}$m$^2$\\
$\alpha_{1}$ & $0.95$ & \\
$\alpha_{2}$ & $0.12$ & \\
\multirow{2}{*}{$\beta$} & $5.0*10^{-10}$ & ~~N$^{-1}$m$^2$s$^{-1}$ \\
& $1.0*10^{-8}$ & ~~N$^{-1}$m$^2$s$^{-1}$ \\
$K_1$ & $6.18*10^{-12}$ & N$^{-1}$m$^4$s$^{-1}$ \\
$K_2$ & $2.72*10^{-11}$ & N$^{-1}$m$^4$s$^{-1}$
\end{tabular}
\end{center}
\end{table}

Figures~\ref{figure_Barenblatt1}--\ref{figure_Barenblatt3} present a comparison between the preconditioned GMRES algorithm, 
the fixed-stress split 
algorithm as presented in~\cite{Hong2018parameter-robust} with a tuning parameter $L=1/(1+\lambda)$ and Algorithm~\ref{alg1}.  
As can be seen, from Figures~\ref{figure_Barenblatt1} and \ref{figure_Barenblatt3} for $\lambda$ being sufficiently large, 
the Uzawa-type method shows similar convergence behaviour to the preconditioned GMRES and fixed-stress methods.

Furthermore, all the numerical results included in Figures~\ref{figure_Barenblatt1}--\ref{figure_Barenblatt3} demonstrate the 
robust performance of the Uzawa-type algorithm
with respect to mesh refinements and variation of the hydraulic conductivities 
$K_1$ and $K_2$, and $\lambda$. 

\begin{figure}
\includegraphics[scale=0.5]{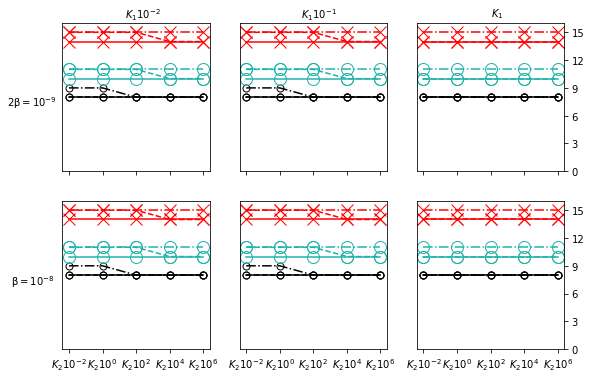}
\caption{Number of preconditioned GMRES (small black circles), fixed-stress split (big green circles) and 
augmented Uzawa-type (red crosses) iterations for preconditioned residual reduction by a factor $10^8$ when solving the 
Barenblatt problem, $\lambda = \widehat{\lambda}$. 
These tests have been performed for $h=1/16$ (dash-dotted line), $h=1/32$ (dashed line) and $h=1/64$ (full line).}
\label{figure_Barenblatt1}
\end{figure}

\begin{figure}
\includegraphics[scale=0.5]{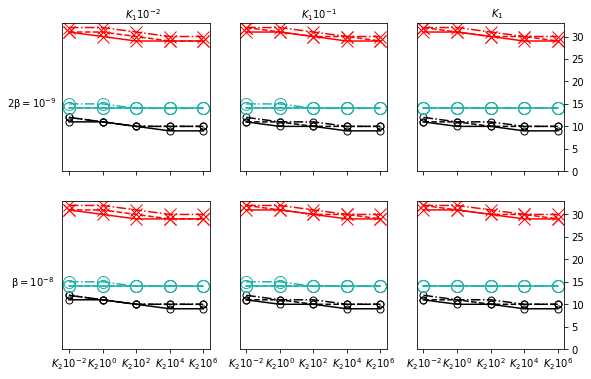}
\caption{Number of preconditioned GMRES (small black circles), fixed-stress split (big green circles) and 
augmented Uzawa-type (red crosses) iterations for preconditioned residual reduction by a factor $10^8$ when solving the 
Barenblatt problem, ${\lambda}:=0.01* \widehat{\lambda}$. 
These tests have been performed for $h=1/16$ (dash-dotted line), $h=1/32$ (dashed line) and $h=1/64$ (full line).}
\label{figure_Barenblatt2}
\end{figure}

\begin{figure}
\includegraphics[scale=0.5]{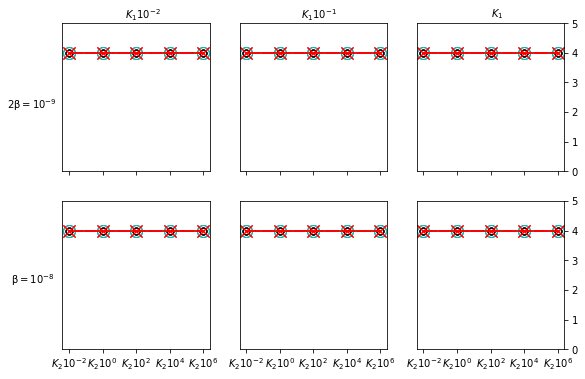}
\caption{Number of preconditioned GMRES (small black circles), fixed-stress split (big green circles) and 
augmented Uzawa-type (red crosses) iterations for preconditioned residual reduction by a factor $10^8$ when solving the 
Barenblatt problem, ${\lambda}:=100* \widehat{\lambda}$. 
These tests have been performed for $h=1/16$ (dash-dotted line), $h=1/32$ (dashed line) and $h=1/64$ (full line).}
\label{figure_Barenblatt3}
\end{figure}

\subsection{The four-network model}

Now we consider system~\eqref{eq:MPET} for $n=4$. This test setting is analogous to the previous example, i.e., 
$\partial \Omega=\bar\Gamma_1 \cup \bar\Gamma_2 \cup \bar\Gamma_3 \cup \bar\Gamma_4$ with
$\Gamma_1$, $\Gamma_2$, $\Gamma_3$, $\Gamma_4$ denoting the bottom, right, top and left boundaries respectively,
$\boldsymbol u=0$~on~$\Gamma_4$, $(\boldsymbol {\sigma} -p_1\boldsymbol I-p_2\boldsymbol I -p_3\boldsymbol I-p_4\boldsymbol I)\boldsymbol n=(0,0)^T$ 
on $\Gamma_1\cup\Gamma_2$, $(\boldsymbol {\sigma} -p_1\boldsymbol I-p_2\boldsymbol I -p_3\boldsymbol I-p_4\boldsymbol I)\boldsymbol n=(0,-1)^T$ on $\Gamma_3$, 
$p_1=2$~on~$\Gamma$, $p_2=20$ on $\Gamma$, $p_3=30$ on $\Gamma$, $p_4=40$ on $\Gamma$. All right-hand sides
have been chosen to be zero.
The reference values of the parameters are taken from~\cite{Vardakis2016investigating} and 
presented in Table~\ref{parameters_MPET4}.

The main aim of the numerical experiments discussed in this subsection is, again, the comparison 
between the three algorithms, namely the preconditioned GMRES algorithm, the fixed-stress split 
algorithm with $L=1/(1+\lambda)$ and the fully decoupling Algorithm~\ref{alg1}. 

Figure~\ref{figure_4network} shows that Algorithm~\ref{alg1} exhibits a convergence behaviour similar to that of the 
preconditioned GMRES method and the fixed-stress split iterative scheme 
over a wide-range of parameters as tabulated. 
%
Moreover, the presented numerical results demonstrate the robustness of the newly proposed algorithm
with respect to large variations of the coefficients $K_3$, $K = K_1 = K_2 = K_4$ and $\lambda$ and the mesh parameter 
$h$.

In order to further compare the performance of the preconditioned GMRES, fixed-stress 
split and augmented Uzawa-type algorithms we present one final table, Table~\ref{tab:times}, with elapsed times measured in seconds. 
These numerical tests have been conducted on a Dell Precision 5540 notebook with an Intel Core i7-9 9850H processor and 64GB RAM.
As the results indicate, the Uzawa-type method is the computationally most efficient among the three, here, clearly seen in terms of 
total solution time when direct methods are used to solve the respective subproblems. A similar behaviour can also be expected 
when iterative solvers of lower complexity replace the direct ones.

\begin{table}[h!]
\caption{Reference values of model parameters for the four-network MPET model.}
\label{parameters_MPET4}
\centering
\begin{tabular}{ccc}
\hline \hline
parameter & value & unit \\[0.0ex] \hline \hline
$\lambda$ & $505$ & Nm$^{-2}$ \\ [0.0ex]
$\mu$ & $216$ & Nm$^{-2}$ \\ [0.0ex]
$c_{p_1}=c_{p_2}=c_{p_3}=c_{p_4}$ & $4.5 *10^{-10}$ & N$^{-1}$m$^2$ \\ [0.0ex]
$\alpha_{1}=\alpha_{2}=\alpha_{3}=\alpha_{4}$  & $0.99$ &\\ [0.0ex]
$\beta_{12}=\beta_{24}$ & $1.5* 10^{-19}$ & N$^{-1}$m$^2$s$^{-1}$ \\ [0.0ex]
$\beta_{23}$ & $2.0*10^{-19}$ & N$^{-1}$m$^2$s$^{-1}$ \\[0.0ex]
$\beta_{34}$ & $1.0* 10^{-13}$ & N$^{-1}$m$^2$s$^{-1}$ \\ [0.0ex]
$K_1=K_2=K_4=K$ & $3.75*10^{-6}$ & N$^{-1}$m$^4$s$^{-1}$ \\ [0.0ex]
$K_3$ & $1.57 * 10^{-9}$ & N$^{-1}$m$^4$s$^{-1}$
\end{tabular}
\end{table}

\begin{figure}
\includegraphics[scale=0.5]{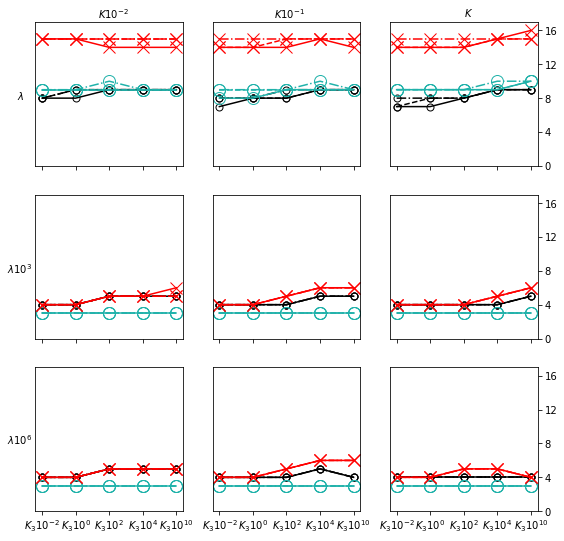}
\caption{Number of preconditioned GMRES (small black circles), fixed-stress split (big green circles) and 
augmented Uzawa-type (red crosses) iterations for preconditioned residual reduction by a factor $10^8$ when solving the 
four-network MPET problem. 
These tests have been performed for $h=1/16$ (dash-dotted line), $h=1/32$ (dashed line) and $h=1/64$ (full line).}
\label{figure_4network}
\end{figure}

\begin{table}[hbtp!]
\caption{Computational times in seconds for the preconditioned GMRES~($t_G$), fixed-stress split ($t_F$) and 
augmented Uzawa-type~($t_U$) algorithms to 
reach preconditioned residual reduction by a factor $10^8$ in
the norm induced by the preconditioner when solving the four-network MPET problem on a mesh with $h=1/64$.
}
\label{tab:times}
\centering
\resizebox{\textwidth}{!}{
\begin{tabular}{c|c||rrr|rrr|rrr|rrr|rrr}
  &  & \multicolumn{3}{c}{$K_3\cdot 10^{-2}$} & \multicolumn{3}{c}{$K_3$} &  \multicolumn{3}{c}{$K_3\cdot 10^{2}$} 
 &  \multicolumn{3}{c}{$K_3\cdot 10^{4}$}  &  \multicolumn{3}{c}{$K_3\cdot 10^{10}$} \\ \hline \hline
\rowcolor{black!10}
 & & $t_{G}$ & $t_{F}$ & $t_{U}$ & $t_{G}$ & $n_{F}$ & $t_{U}$ & $t_{G}$ & $t_{F}$ & $t_U$
& $t_{G}$ & $t_{F}$ & $t_U$ & $t_{G}$ & $t_{F}$ & $t_U$
\\ 
\hline
\multirow{3}{*}{$\lambda$}  & $K\cdot 10^{-2}$ 
                              & 15.54 &  \color{orange}{8.98} & \color{c}7.26  
                              & 15.39 & \color{orange}{9.12} & \color{c}7.21 & 15.51 & \color{orange}{9.09} & \color{c}7.16 
                              & 15.90 & \color{orange}{9.17} & \color{c}7.21 & 15.83 & \color{orange}{9.33} & \color{c}7.24 \\ [0.5ex]
                                 & $K$  & 15.32 & \color{orange}{9.20} & \color{c}7.68 & 15.60 & \color{orange}{9.13} & \color{c}7.66 
                                 & 15.40 & \color{orange}{8.91} & \color{c}7.13 
                                & 15.75 & \color{orange}{9.12} & \color{c}7.41 & 16.09 &\color{orange}{9.26}  & \color{c}7.68   \\[0.5ex]
                                & $K\cdot 10^{2}$ & 15.25 &  \color{orange}{9.17} & \color{c}7.53 & 15.47 &\color{orange}{9.27} & \color{c}7.19 
                               & 15.24 & \color{orange}{9.08} & \color{c}7.52 
                               & 15.44 & \color{orange}{9.11} & \color{c}7.28 & 15.64 & \color{orange}{9.17} & \color{c}7.37   \\ [0.5ex] \hline
 \multirow{3}{*}{$\lambda\cdot 10^3$}  
                             & $K\cdot 10^{-2}$ & 14.87 &  \color{orange}{7.80} & \color{c}5.45 & 15.00 &  \color{orange}{7.74} & \color{c}5.68 
                             & 14.95 &  \color{orange}{7.56} & \color{c}5.93 
                             & 15.16 &  \color{orange}{8.05} & \color{c}5.48 & 15.31 &  \color{orange}{8.64} & \color{c}6.10  \\ [0.5ex]
                                     & $K$  & 14.71 &  \color{orange}{7.78}  & \color{c}5.38 & 14.81 &  \color{orange}{7.91} & \color{c}5.42 
                                     & 14.74 &  \color{orange}{8.10} & \color{c}5.75  
                                    & 15.05 &  \color{orange}{8.03} & \color{c}6.68  & 15.23 &  \color{orange}{8.07} & \color{c}6.05  \\ [0.5ex]
                                     & $K\cdot 10^{2}$ & 14.92 &  \color{orange}{8.91} & \color{c}6.78  & 14.97 &  \color{orange}{8.92} & \color{c}6.69  
                                     & 14.90 &  \color{orange}{8.80} & \color{c}6.96 
                                     & 14.83 &  \color{orange}{8.64} & \color{c}5.21 & 14.87 &  \color{orange}{9.27} & \color{c}5.28    \\ [0.5ex] \hline
 \multirow{3}{*}{$\lambda\cdot 10^6$}  & $K\cdot 10^{-2}$ & 14.98 &  \color{orange}{8.95} & \color{c}6.14 & 15.02 &  \color{orange}{9.06} & \color{c}7.07  
                                   & 14.81 & \color{orange}{7.67} & \color{c}7.12
                                  & 14.75 &  \color{orange}{7.65} & \color{c}5.90 & 14.96 & \color{orange}{7.53} & \color{c}5.81  \\ [0.5ex]
                                    & $K$  & 14.91 &  \color{orange}{8.89} & \color{c}5.40  & 15.08 & \color{orange}{8.96} & \color{c}5.61 & 15.12 &  \color{orange}{9.03} & \color{c}7.01
                                   & 15.06 &  \color{orange}{9.12} & \color{c}6.33 & 15.19 &  \color{orange}{9.32}  & \color{c}6.20   \\ [0.5ex]
                                   & $K\cdot 10^{2}$ & 14.72  &  \color{orange}{9.27} & \color{c}4.92 & 14.88 &  \color{orange}{8.99}  & \color{c}5.00 
                                   & 15.09 &  \color{orange}{8.96}  & \color{c}5.26
                                  & 15.52 &  \color{orange}{9.19} & \color{c}5.40  & 15.12 &  \color{orange}{9.24} & \color{c}5.54                               
\end{tabular}
}
\end{table}

\subsection{Scaling test} 

Finally, we present a scaling test demonstrating the convergence behaviour of the preconditioned GMRES and 
augmented Uzawa-type algorithms with respect to the number of fluid networks $n$. These methods have been tested for 
$n=1,2,4,8$. 

In order to perform a reasonable comparison, we have assumed that 
all network transfer coefficients are equal to~$0$ irrelevant to the number of networks, 
$\lambda = 10^3$, $R_i^{-1}=10^4$, $\alpha_{p_i}=10^{-4}$, $i=1,\ldots,n$. 
The test setting is similar to those of the previously considered Biot-Barenblatt and four-network models, i.e., 
$\partial \Omega=\bar\Gamma_1 \cup \bar\Gamma_2 \cup \bar\Gamma_3 \cup \bar\Gamma_4$ with
$\Gamma_1$, $\Gamma_2$, $\Gamma_3$, $\Gamma_4$ being the bottom, right, top and left boundaries respectively,
$\boldsymbol u=0$~on~$\Gamma_4$, 
$(\boldsymbol {\sigma} -\sum_{i=1}^n p_i\boldsymbol I)\boldsymbol n=(0,0)^T$ 
on $\Gamma_1\cup\Gamma_2$, $(\boldsymbol {\sigma} -\sum_{i=1}^n p_i\boldsymbol I)\boldsymbol n=(0,-1)^T$ on $\Gamma_3$ 
and $p_i=10$, $i=1,\ldots,n$~on~$\Gamma$. As previously, all the right-hand sides
have been chosen to be zero. 

We have conducted the numerical tests on a mesh with a mesh-size $h=1/32$. In all test cases, the number of 
required iterations to reach a preconditioned residual reduction by a factor $10^8$  
equals $4$. This clearly indicates the robustness of the proposed algorithms with respect to the number of networks as suggested 
by our theoretical findings. 

\section{Concluding Remarks} 

The main contribution of this manuscript is the development of a new augmented Lagrangian Uzawa algorithm for 
three-by-three double saddle point block systems arising in Biot's and multiple network poroelasticity models. 
The proposed method fully decouples the fluid velocity, fluid pressure and solid displacement fields, 
contrary to the fixed-stress iterative scheme, which decouples only the flow from the mechanics problem. 
In this manner the subsystems that need to be solved in every iteration become considerably smaller, especially 
in the models where multiple fluid networks are present.

The presented convergence analysis proves the parameter-robust linear convergence 
of the new algorithm and additionally offers explicit formulas for a proper choice of required stabilization parameters. 
All numerical tests confirm the robustness and efficiency of the new fully decoupled iterative scheme and also its 
superiority in terms of computational work over existing methods.

\bibliographystyle{amsplain}
\bibliography{reference_mpet}

\end{document}